\documentclass[12pt,a4paper,twoside]{amsart}
\usepackage{amsmath}
\usepackage{amssymb}
\usepackage{amsfonts}
\usepackage{a4}

\usepackage{tikz}
\tikzset{node distance=3cm, auto}
 \usepackage[normalem]{ulem}

\newtheorem{theorem}{Theorem}[section]

\newtheorem{lemma}[theorem]{Lemma}

\newtheorem{proposition}[theorem]{Proposition}

\newtheorem{corollary}[theorem]{Corollary}

\theoremstyle{definition}

\newtheorem{definition}[theorem]{Definition}

\newtheorem{claim}[theorem]{Claim}

\newtheorem{example}[theorem]{Example}

\newtheorem{question}[theorem]{Question}

\newtheorem{remark}[theorem]{Remark}

\newcommand{\mA}{\mathbb A}

\newcommand{\mB}{\mathbb B}

\newcommand{\mC}{{\mathbb C}}

\newcommand{\mE}{{\mathbb E}}

\newcommand{\mF}{\mathbb F}

\newcommand{\mP}{\mathbb P}

\newcommand{\mR}{{\mathbb R}}

\newcommand{\mT}{\mathbb T}

\newcommand{\mV}{\mathbb V}

\newcommand{\mW}{\mathbb W}

\newcommand{\mX}{\mathbb X}

\newcommand{\mY}{\mathbb Y}

\newcommand{\mZ}{{\mathbb Z}}

\newcommand{\ho}{\hookrightarrow}

\newcommand{\Gg}{\gamma}

\newcommand{\GG}{\Gamma}

\newcommand{\bo}{\omega}

\newcommand{\bs}{\sigma}

\newcommand{\ep}{\epsilon}

\newcommand{\D}{\Delta}

\newcommand{\kk}{\kappa}

\newcommand{\mcA}{\mathcal A}

\newcommand{\mcB}{\mathcal B}

\newcommand{\mcC}{\mathcal C}

\newcommand{\mcE}{\mathcal E}

\newcommand{\mcF}{\mathcal F}

\newcommand{\mcG}{\mathcal G}

\newcommand{\mcI}{\mathcal I}

\newcommand{\mcL}{\mathcal L}

\newcommand{\mcP}{\mathcal P}

\newcommand{\mcU}{\mathcal U}

\newcommand{\ti}{\tilde}

\newcommand{\lr}{\leftrightarrow}

\newcommand{\emp}{\emptyset}

\newcommand{\bP}{\bar P}
\newcommand{\bd}{\bar d}

\newcommand{\Aut}{\operatorname{Aut}}
\newcommand{\Hom}{\operatorname{Hom}}
\newcommand{\af}{\operatorname{af}}
\newcommand{\sing}{\operatorname{sing}}

\newcommand{\codim}{\operatorname{codim}}

\usepackage{color}
\usepackage{soul}

\begin{document}

\title{Properties of high rank subvarieties of affine spaces}

\begin{abstract} We use tools of additive combinatorics for the study of  subvarieties defined by {\it high rank} families of polynomials in  high dimensional $\mF _q$-vector spaces.
In the first, analytic part of the paper we prove  a number properties of high rank systems
of polynomials. In the second, we use these properties to 
deduce results  in  Algebraic Geometry , such as an effective Stillman conjecture over algebraically closed fields, an analogue of Nullstellensatz for varieties over finite fields, and a strengthening of a recent result of \cite{Jan}. 
We also show that for $k$-varieties $\mX \subset \mA ^n$ of high rank any weakly polynomial 
function on  a set $\mX (k)\subset k^n$ extends to a polynomial.

 \end{abstract}

\author{David Kazhdan}
\address{Einstein Institute of Mathematics,
Edmond J. Safra Campus, Givaat Ram 
The Hebrew University of Jerusalem,
Jerusalem, 91904, Israel}
\email{david.kazhdan@mail.huji.ac.il}

\author{Tamar Ziegler}
\address{Einstein Institute of Mathematics,
Edmond J. Safra Campus, Givaat Ram 
The Hebrew University of Jerusalem,
Jerusalem, 91904, Israel}
\email{tamarz@math.huji.ac.il}

\thanks{The second author is supported by ERC grant ErgComNum 682150. 
}

\maketitle

\section{ Introduction}

\bigskip

Let $k$ be a field. For an algebraic $k$-variety $\mX$ 
we write $X(k):=\mX (k)$. To simplify notations we often write $X$ instead of $X(k)$. In particular we write $V:=\mV (k)$ when $\mV$ is a vector space and write $k^N$ for $\mA ^N(k)$. 

For a $k$-vector space $\mV$ we denote by $\mcP_d(\mV)$ the algebraic variety of polynomials on $\mV$ of degree $\leq d$ and  by $\mcP_d(V) $ the set of polynomials functions $P:V\to k$ of degree $\leq d$.  We always assume that   $d<|k|$, so the restriction  map $\mcP_d(\mV)(k)\to \mcP_d(V)$ is a bijection. For a family $\bar P=\{ P_i\}$ of polynomials on $\mV$
we denote by $\mX _{\bar P}\subset \mV$ the subscheme  defined by the ideal generated by $\{ P_i\} $ and by  $X_ {\bar P} $ the set $ \mX _P (k)\subset V$. We will not distinguish between the set of affine $k$-subspaces of $\mV$ and the set of affine  subspaces of $V$ since for an affine $k$-subspace $\mW \subset \mV$, the map $\mW \to \mW (k)$ is a bijection.

In the introduction we consider only the case of hypersurfaces $\mX \subset \mV$ and provide an informal outline of main results. Precise definitions appear in the next section.

\begin{definition}Let $P$ be a polynomial of degree $d$ on a  $k$-vector space $V$.
\begin{enumerate} 
\item  We denote by $\ti P :V^d\to k$ the multilinear symmetric form associated with $P$ defined by 
$\ti P(h_1, \ldots, h_d) : =  \Delta_{h_1} \ldots  \Delta_{h_d} P: V^d \to k$, where $\Delta_hP(x) = P(x+h)-P(x)$.
\item  The {\em rank} $r(P)$ is  the minimal number $r$ such that $P$ can be written in the form $P=\sum _{i=1}^rQ_iR_i$, where $Q_i,R_i$ are polynomials on $V$ of degrees $<d$. 
\item  We define the  {\em non-classical rank (nc-rank)} $r_{nc}(P)$ to be the rank of $\ti P$.
\item A polynomial $P$ is {\it $m$-universal} if for any polynomial $Q\in \mcP_d(k^m)$ of degree $d$ there exists an affine
 map $\phi :k^m\to V$ such that $Q=P\circ \phi$.
\item We denote by $\mX _P\subset \mV$ the hypersurface defined by the equation $P(v)=0$ and by $\mX _P^{\sing}$ the singular locus of $\mX _P$.
\item $s(P):= \codim _{\mX _P}( \mX _P^{\sing}) $.
\end{enumerate}
\end{definition}

\begin{remark}\label{pd} 
\begin{enumerate}
\item If char$(k) >d$ then $r(P) \le r_{nc}(P)$.
\item 
In low characteristic it can happen that $P$ is of high rank and $\ti P$ is of low rank, for example in characteristic $2$ the polynomial $P(x) = \sum_{1 <i<j<k <l\le n} x_ix_jx_kx_l$ is of rank $\sim n$, but of nc-rank $3$, see \cite{gt1}, \cite{lms}, \cite{tz-1}.
\end{enumerate}
\end{remark}

\begin{remark}\label{uniform-remark} 
 Let $\mcF _Q$ be the set of all affine maps $\phi :k^m\to V$ such that 
$Q=P\circ \phi$. We later show that in the case when $r_{nc}(P) \gg1$, the size 
of the set  $\mcF _Q$ is essentially independent of the choice of  
$Q\in \mcP_d(k^m)$.  
\end{remark}

We now turn to the applications.
\subsubsection{Sections of high rank varieties}

\begin{theorem} [Acc]\label{Acc} There exists $C=C(d)$ 
with the following property. For any field $k$ which is  either algebraically closed or finite field, any $k$-vector space $V$ and any polynomial $P\in \mcP_d(V)$ of nc-rank $>Cm^C$ the following hold: 
\begin{enumerate}
\item Any polynomial $Q\in \mcP_d(k^m)$ is $m$-universal.
\item Let $\text{Aff} _m(\mV) $ be the variety of affine maps $\phi :\mA ^m\to \mV$ and $\ti \kappa _P :  \text{Aff} _m(\mV) \to \mcP_d(\mA ^m) $ be the algebraic map defined by $\ti \kk _{P} (\phi):= P \circ \phi$. Then 
all fibers of  $\ti \kappa _P$ are algebraic varieties of the same dimension.
\item The map $\ti \kappa _P :  \text{Aff} _m(\mV) \to \mcP_d(\mA ^m) $ is flat.
\end{enumerate}
\end{theorem}

\begin{remark} 
\begin{enumerate}
\item  It is easy to see that $s(P)\leq 2r(P)$.
\item In the case when $k$ is an algebraically closed field of characteristic $0$ it is shown in \cite{S} that 
  $r(P)\leq d!s(d)$.
\item We can  show  the existence of $\ep =\ep (d) >0$ such that $r_{nc}(P)\leq \ep s^\ep(P)$ for polynomials of degree $d$.
\end{enumerate}
\end{remark}

\begin{corollary}There exist $C=C(d)$
such that  for any field $k$ which is  either algebraically closed or finite, a $k$-vector space $V$, any polynomial $P\in \mcP_d(V)$ such that $s(\tilde P)\geq Cm^C$ is $m$-universal.
\end{corollary}
\begin{question} Does there exist $\delta =\delta (d)$ such that $r(P)\leq \delta s(P)^\delta $?
\end{question}

\begin{remark}The analogous result holds for a system $\bar P$ of polynomials.
\end{remark}

\subsubsection{A strengthening of the main Theorem from \cite{Jan}. } In \cite{Jan} authors show that any non-trivial Zariski-closed condition on tensors that is functorial in the underlying vector space implies bounded rank. We show that the condition of being  Zariski-closed can be omitted.

\begin{theorem}\label{Jan} Let $k$ be an algebraically closed field, $\mcC$ the category of finite-dimensional affine $k$-vector spaces with morphisms being affine maps, let   $\mcF _d$ be the contravariant endofunctor  on $\mcC$ given by  
$$\mcF _d(V)=\{\text{Polynomials on $V$ of degree $\leq d$} \},$$ 
and let  $\mcG \subset \mcF$ be a proper subfunctor. Then there exists $r$ such that $r_{nc}(P)\leq r$ for any finite-dimensional $k$-vector space $V$ and $P\in \mcG (V)$.
\end{theorem}

\subsubsection{Extending weakly polynomial functions}

Let $k$ be a field, $V$ a $k$-vector space and $X$ a  subset of $V $. 
A function  $f:X\to k$ is {\em weakly polynomial} of degree $\leq a$, if the  restriction of $f$ on any  affine subspace $L$ of $V$ contained in $X$ is a  polynomial of degree $\leq a$.  
\begin{theorem}\label{ext}
Let $k$ be a field which either an algebraically closed field, or a finite field such that $|k|>ad$ and let $X\subset V$ be a hypersurface defined by a polynomial of degree $d$ of sufficiently high nc-rank. Then 
any $k$-valued weakly polynomial function of degree $\leq a$  on $X$ is a restriction of a polynomial  $F$ on $V$ of degree $\leq a$. 
\end{theorem}

\begin{remark}The main difficulty in a proof of Theorem \ref{ext} is the non-uniqueness of $F$ in the case when $a>d$.
\end{remark}

\subsubsection{Nullstellensatz over $\mF _q$.}
We prove the following variant of  Nullstellensatz for polynomials over $\mF_q$.  
\begin{theorem}\label{N}
There exists  $r(d)$ such that for any finite field $k=\mF _q,$  a $k$-vector space $V$ and a  
$k$-polynomial $P$ of degree $d$ and  nc-rank larger than $r(d)$ the following holds:
Any   polynomial $R$ of degree $<q/d$ vanishing at all points $x\in X_P$ 
is divisible by $P$.
\end{theorem}

\begin{remark} 
\leavevmode
\begin{enumerate}
\item This result is a strengthening of the Proposition 9.2 from  \cite {gt1} in two ways:
   \begin{enumerate}
     \item We show the independence of $r(d)$ on the degree of $R$, and 
     \item The paper \cite{bl} shows the existence of polynomials $Q_i$ of bounded degrees such that $R(x)=\sum _{i=1}^cQ_i(x)P_i(x)$, for all 
           $x\in \mX (k)$, but does not show that
$R$ is contained in the ideal generated by $\{ P_i\}$.
    \end{enumerate} 
\item We outline ideas of  proofs of Theorems \ref{ext}, \ref{N} at the beginning of corresponding sections. 
\item
The quantitative bound on the rank in all proofs depend {\em only} on the bounds in Theorem \ref{bias-rank-1}. We conjectured 
that the bound on $r$ in Theorem \ref{bias-rank-1} depends polynomially on $s$. This conjecture was proved in \cite{Mi} (and  in a slightly weaker form in \cite{janzer}).
\end{enumerate}
\end{remark}

\section{The formulation of results.}
We start with a series of definitions.

\begin{definition}[Rank]\label{rank} 
Let $k$ be field and $V$ a $k$-vector space.
\begin{enumerate}
\item A tensor of degree $d$ on $V$ is a multilinear map $T:V^d\to k$.
 \item Let $P:V \to k$ be a polynomial of degree $d$. We define the {\em rank} $r(P)$ as the minimal 
number $r$ such that $P$ can be written as a sum $P=\sum _{j=1}^rQ_jR_j$ where $Q_j,R_j$ are 
polynomials of degrees $<\deg(P)$ defined over $k$ \footnote{This notion of rank is also known as {\em Schmidt-rank} in the analytic number theory literature or {\em strength} in the algebraic geometry literature.}. We define the rank of a degree $1$ polynomial to be infinite, and the rank of the zero polynomial to be $0$. 

 \item For a polynomial $P:V \to k$ of degree $d$ we define $\ti P(h_1, \ldots, h_d) : =  \Delta_{h_1} \ldots  \Delta_{h_d} P: V^d \to k$, where $\Delta_hP(x) = P(x+h)-P(x)$. 
 \item  We define the  {\em non-classical rank (nc-rank)} $r_{nc}(P)$ to be the rank of $r(\ti P)$.
 \item For a $d$ tensor $P : V^d \to k$ we define the {\em partition rank (p-rank)}  $pr(P)$ as the minimal 
number $r$ such that $P$ can be written as a sum $P(x_1, \ldots, x_d)=\sum _{ J_i \subset [1,d], i \in [r]} Q_j((x_l)_{l \in J_i})R_j((x_l)_{l \in J^c_i})$ where $Q_j,R_j$ are degree $<d$ tensors.

\item For any sequence $\bar d=(d_1,\dots ,d_c)$ we denote by $\mcP _{\bar d}(V)$ the space of families
 $\bar P=(P_i)_{ 1\leq i\leq c} $  of polynomials such that $\deg P_i \le d_i$
We define the {\em rank} $r(\bar P)$ as the minimal rank of the polynomials $\sum _{\bar a\in k^c \setminus \{0\}}a_iP_i$. 
We denote $d:= \max d_i$. Define $r_{nc}(\bar P)$, $pr(\bar P)$ similarly. 
 \item For a family $\bar P=(P_i)_{1\leq i\leq c}$ of polynomials on $\mV$ we define the subscheme $\mX _{\bar P}\subset \mV$ by the system of equations $\{v \in V: P_i(v)=0, 1\leq i\leq c\}$,  and define the  rank of   $\mX _{\bar P}$ to be the rank of $\bar P$.
\end{enumerate}
\end{definition}

\begin{remark} The rank $r(\bar P)$ depends only on  the linear span  $L(\bar P)$ of $(P_i)$.
\end{remark}

The following statement is immediate.

\begin{claim}\label{ord}Let $\bar P=( P_i)_{i=1}^c$ be 
 a family of polynomials with  $r(\bar P)\neq 0$. Then
\begin{enumerate} 
\item $\dim(L(\bar P))=c$.
\item There exists a basis $Q_i^j$ in $L(\bar P)$ such that 
\begin{enumerate} 
\item $\deg (Q^j_i)=j$.
\item For any $j$ the family $\bar Q^j=(Q_i^j)$ is of rank $\geq r(\bar P)$.
\end{enumerate}
\end{enumerate}
\end{claim}

\begin{remark} Since the subscheme $\mX _{\bar P} $ depends only on the space $L(\bar P)$,  we can (and will)  assume that all the families $\bar P$ we consider satisfy the conditions of Claim \ref{ord} on $\bar Q$.
We also  may  assume that $\bar P$ has no affine polynomials, since we can initially restrict to the locus of the affine polynomials. 
\end{remark}

\begin{lemma}\label{rank-p-rank}  Let $P$ be a $d$-tensor.  Then  $r(P)\le pr(P) \le 4^dr(P)$.
\end{lemma}
\begin{proof}
 Let $P$ be a $d$-tensor. 
It is clear that $r(P)\le pr(P)$. We show that $pr(P) \le 4^dr(P)$.   Let $k$ be a field with  $|k|\geq d+2$ and write $G=(k^\star)^d$. For any 
$\bar j =\{ j_1,\dots ,j_d\}$, $j_d\leq d$ we denote by $\chi _{\bar j}:G\to k^\star$ the character $\bar t\to \prod _{i=1}^dt_i^{j_i}$ where $\bar t =(t_1,\dots,t_d)$.

Let $V_i$, $1\leq i\leq d$ be $k$-vector spaces and $\mcP$ the space of polynomials $P: V_1\times \ldots \times  V_d \to k$ of degree $\leq d$. Let $J$ be the 
space of maps $\bar j$ from $ [1,d]^d \to [1,d] $.
For any $\bar j\in J$ we denote by  $\mcP _{\bar j}$ the subspace of 
polynomials $P: V_1\times \ldots \times V_d\to k$ such that 
$$P(t_1v_1, \dots ,t_dv_d)= \prod _{i=1}^dt_i^{j_i} P(v_1, \dots ,v_d).$$ 
Since $|G|$ is prime to char$(k)$ and  $|k|\geq d+2$ we have a direct sum 
decomposition 
$ \mcP =\oplus _{\bar j\in J} \mcP _{\bar j} $. We denote by $p_{\bar j} : \mcP \to \mcP _{\bar j} $ the projection. 

We write $\bar 1=(1,\dots ,1)\in J$ denote by $p_d : \mcP \to \mcP _{\bar 1} $ the projection.  For any $a,b, a+b=d$ we denote by 
$J_{a,b}\subset J\times J$ the set of pairs $\bar j_a, \bar j_b$ such that 
$\text{supp} ( \bar j_a)\cap \text{supp} ( \bar j_b)=\emp$, $\deg ( \bar j_a)=a$, $\deg ( \bar j_b)=b$. 

For any polynomials $Q,R \in \mcP $ of degrees  $a,b, a+b=d$ we write 
$P(Q,R)=\sum _{ (j_1,j_2)\in J_{a,b}} p_{ j_1}(Q) p_{ j_2}(R)$. We have that  $p_{\bar 1}(QR)= P(Q,R) $, and therefore  $pr(P)\leq 4^dr(P)$.
\end{proof}

\begin{definition}\label{Af}
\leavevmode
\begin{enumerate}
\item For $m\geq 1$ and a $k$-vector space $V$, we denote by $\text{Aff} _m(\mV)$ the algebraic variety of affine maps $\phi :\mA ^m\to \mV$ and write $\text{Aff} _m(V) := \text{Aff} _m(\mV) (k)$.
 \item We define an algebraic morphism $\ti \kk _{\bar P}: \text{Aff} _m(\mV) \to \mcP_{\bar d}(\mA ^m)$ by $\ti \kk _{\bar P} (\phi):= \bar P \circ \phi$, and denote  by $\kk _{\bar P}$ the corresponding map  $\text{Aff} _m(V) \to \mcP_{\bar d}(k ^m) $.
\end{enumerate}
\end{definition}

\begin{definition}\label{un}
 A map $\kk :M\to N$ between finite sets is {\em $\ep$-uniform}, where $\ep >0$, if for all $n\in N$ we have
$$\left||N||\kk ^{-1}(n) |-|M| \right|\leq \ep |M|.$$
\end{definition}

\begin{remark} In this paper we say that a bound $r(\bar d,m,t)$ is {\it effective} if 
\begin{enumerate}
\item For  fixed  $d:=\max _id_i$, the bound 
is polynomial in $m, t, c$.  
\item The dependence on $d$ is doubly exponential.
\end{enumerate}
 The  effective lower bounds for  $r(\bar d,m,t)$ follow from the Conjecture \ref{conj-bias} proven  in \cite{Mi}.
\end{remark}

\begin{theorem}\label{A1} For any sequence $\bar d$ and any $m, t\geq 1$, there exists an effective bound $r(\bar d,m,t)$ such that 
 for any finite field $k=\mF _q$, a $k$-vector space $V$ and
 a family $\bar P\in \mcP_{\bar d}(V)$ of nc-rank $\geq r(\bar d,m,t)$, the map $\kk _P$ is $q^{-t}$ uniform.
\end{theorem}

To formulate Remark \ref{uniform-remark} precisely we introduce a number of additional definitions.

\begin{definition}\label{b1}  Fix  an affine hypersurface $W\subset V$. We denote by $\mZ _{\bar P}$ the variety  of affine $m$-dimensional subspaces $L \subset X _{\bar P}\cap W$ and by 
$\mY  _{\bar P}\subset \mZ  _{\bar P} $ the subvariety consisting of $L \subset \mX _{\bar P}\cap W$ such that there is no $(m+1)$-dimensional affine subspace  $M\subset \mX _{\bar P}, M\not \subset W$ containing $L$. 
\end{definition}

\begin{theorem}\label{B} For any $\bar d$, any $m,t\geq 1$, there exists an effective bound $\ti r(m,t, \bar d)$ with the following property.
 For any finite field $k=\mF _q$, any $k$-vector space $V$ and any family  $\bar P\in \mcP_{\bar d}(V)$ of  nc-rank $\geq \ti r(m,t, \bar d)$, we have that
$\frac{|Y_{\bar P}|}{|Z_{\bar P}|}\leq q^{-t}$.
\end{theorem}

\begin{remark} 
The condition that $r(P)\gg1$ does not imply that $Y _{\bar P} =\emp$.
To see that $r(P)\gg1$ does not imply the emptiness of the set  $Y _{\bar P} $ consider the case
 $V=k^n,W=\{w \in V: w_n=0\}, m=1, P=\sum _{i=1}^ {n-1}x_i^d+x_n$ and $L=ke_1$.
 \end{remark}

\subsection{Applications} In this subsection we provide precise formulations of Theorems \ref{Acc}, \ref{ext} and \ref{N}.

\subsubsection{The surjectivity over algebraically closed fields} 
We start with  a formalization of Theorem \ref{Acc}.

\begin{theorem} [Effective Stillman conjecture] \label{AC} There exists an effective bound $r(m,t, \bar d)$ with the following property.

For any sequence $\bar d$ and  $m,t\geq 1$,  any algebraically closed  field $k$,  any $k$-vector space $V$ and any family $\bar P\in \mcP_{\bar d}(V)$ of nc-rank $\geq r(m,t, \bar d)$ the following holds
\begin{enumerate}
\item The map $\kk _{\bar P}$ is surjective.
\item All fibers of the morphism  $\ti \kk _{\bar P}$ are of the same dimension.
\item The morphism  $\ti \kk _{\bar P}$  is flat.
\item $\bar P\subset k[V^\vee]$ is a regular sequence.
\end{enumerate} 
\end{theorem}
\begin{remark}\leavevmode
\begin{enumerate}
\item Parts (1) and (2) follow from Theorem \ref{A1}.
\item The part (3) follows from the parts (1),(2) and  Theorem 23.1 in [22]. The part (4)
is an immediate consequence of the part (3).

\item As shown in \cite{ah} the part (4) of the theorem implies the validity of an effective Stillman conjecture.
\end{enumerate}
\end{remark}

In the formulation of the next result we use notation introduced in  Definition \ref{b1}.

\begin{theorem}\label{BC}For any $\bar d$, $m,t\geq 1$, there exists effective bound $\ti r(m,t, \bar d)$ with the following property.
 For any algebraically closed  field $k$, a $k$-vector space $V$ and a family  $\bar P\in \mcP_{\bar d}(V)$  of nc-rank $\geq \ti r(m,t, \bar d)$, we have that
 $\dim (\mZ _{\bar P})-\dim(\mY _{\bar P})\geq t$.
\end{theorem}

\subsubsection{Extending weakly polynomial functions}
\begin{definition}\label{weak-def}\leavevmode
\begin{enumerate}
\item  Let  $V$ be a $k$-vector space  and  $X\subset V$. We say that a function $f:X \to k$ is {\it weakly polynomial} of degree $\leq a$ if the restriction $f_{|L}$ to any affine subspace  $L \subset X$ is a polynomial of degree $\leq a$.  
\item $X$ satisfies $\star ^a$ if any weakly polynomial function of degree $\leq a$ on $X$ is a restriction  of a polynomial function of degree $\leq   a$ on $V$.
\end{enumerate}
\end{definition}

Below is a formalization of Theorem \ref{ext} from the introduction. 

\begin{theorem}\label{w}For any $\bar d$ and $a\geq 1$ there exists an effective bound $r(\bar d,a)$ such for any field $k$ which is either finite with $|k|>ad$ or algebraically closed, a $k$-vector space $V$ and a family $\bar P\in \mcP _{\bar d}(V)$ of nc-rank $\geq r(\bar d,a) $
the subset $X_{\bar P}\subset V$ has the property $\star _a$.
\end{theorem}

\subsubsection{Nullstellensatz over $\mF _q$}

Below is a formalization of Theorem \ref{N} from the introduction.  

\begin{theorem}\label{main-null-int} For any sequence  $\bar d$, there exists  an effective bound $r(\bar d)>0$ such that the following holds. Let $k=\mF _q$ be a finite field and let $V$ be a $k$-vector space and $\bar P=\{ P_i\}$ be a family of $k$-polynomials of degrees $\leq d_i$ on $V$ of nc-rank greater than $r(\bar d)$. Then  any  polynomial $R$ of degree $ <q/\tilde d, \ \tilde d:=\prod _{i=1}^cd_i,$ such that $R(x)=0$ for each $x\in \mX _{\bar P}(k)$, belongs to the ideal $J(\bar P):=(P_1,\dots ,P_c)$.  
\end{theorem}

\subsection{Acknowledgement}
We thank U. Hrushovski for his help with simplifying the proof of Theorem \ref{main-null-int}.


\section{Analysis}

In the main part of this  section we  prove Theorems \ref{A1} and \ref{B} using the results on equidistribution of high rank families of polynomials. These are based on the 
technique of the additive combinatorics. At the end of the section we apply  these results to  prove Theorems \ref{AC}, \ref{BC} and  \ref{Jan}.

\subsection{ Equidistribution of high rank families of polynomials} 

The most basic result is the following proposition 
on equidistribution of high rank families of polynomials:
\begin{proposition}\label{size}  
For any sequence $\bar d =(d_1,\dots ,d_c)$ and any $s\geq 1$ there exists an effective bound $r(\bar d,s)$ with the following property.
 For any finite field $k=\mF _q$, a $k$-vector space $V$ and a family $\bar P\in \mcP_{\bar d}(V)$ of nc-rank $\geq r(\bar d, s)$ the map $\bar P:V\to k^c$ is $q^{-s}$-uniform.
\end{proposition} 

The main ingredient of this proof  comes from the relation between the bias of exponential sums and algebraic rank.

Let $k$ be a finite field, char$(k)=p$, $|k|=q$. Let  $V$ be a vector space over $k$. We denote $e_q(x) = e^{2 \pi i \psi(x)/p}$, where $\psi:k \to \mathbb F_p$ is the trace function. Then $e_q$ is a non trivial additive character. 
Let $P :V \to k$ be a polynomial of degree $d$. Recall that $\ti P(h_1, \ldots, h_d)== \Delta_{h_d} \ldots  \Delta_{h_1}P(x)$ is the multilinear form on $V^d$ associated with $P$, and we can write 
 \[
\ti P(h_1, \ldots, h_d)= \sum_{\bo \in \{0,1\}^d}( -1)^{|\bo|}P(x+\omega \cdot \bar h); \qquad |\bo| = \sum_{i=1}^d \bo_i, \quad 
\omega \cdot \bar h =  \sum_{i=1}^d \bo_ih_i.
\]

We denote by $\mE_{x \in S} f(x)$ the average  $|S|^{-1}\sum_{x \in S} f(x)$. 

\begin{definition}[Gowers norms \cite{gowers}]\label{uniform} For a function $g: V \to \mC$ we define the norm $\|g\|_{U_d} $ by \[\|g\|^{2^d}_{U_d} = \mE_{x,v_1, \ldots v_d\in V} \prod_{\omega \in \{0,1\}^d} g^{\omega}(x+\omega \cdot \bar v),\] where $g^{\omega}=g$  if $|\omega|$ is even and $g^{\bo}=\bar g$ otherwise.
 \end{definition}

\begin{definition}[Analytic rank] The analytic rank of a polynomial $P:V \to k$ of degree $d$ is defined by 
arank$(P)=-\log_q \|e_q(P)\|_{U_d}$. 
 \end{definition}

The key analytic tool in this paper is the following theorem relating bias and rank.

 \begin{theorem}[Bias-rank]\label{bias-rank-1} 
 \leavevmode
 \begin{enumerate}
\item Let $s,d>0$. There exists $r=r(s,d)$ such that for any finite field $k$ of size $q$, any vector space $V$ over $k$, any polynomial  $P:V \to k$ of degree $d$ the following holds. If $P$ is of nc-rank $>r$ then
\[
\|e_q(P)\|^{2^d}_{U_d} = |\mE_{v \in V}e_q(\ti P(h_1, \ldots, h_d))| < q^{-s}.
\]
\item  Let $r,d>0$. For any finite field $k$ of size $q$, any vector space $V$ over $k$, any polynomial  $P:V \to k$ of degree $d$, if
\[
 |\mE_{h_1, \ldots, h_d} e_q( \ti P(h_1, \ldots, h_d)| <q^{-r}
\]
for some polynomial, then $P$ is of partition rank $>r$.
\end{enumerate}
\end{theorem}
 
\begin{proof}
Part (1) was proved  for partition rank in increasing generality in \cite{gt1,kl,bl}. The most general version of the first part can be found at the survey \cite{hhl} (Theorem 8.0.1). Part (2) for was observed in \cite{kz-approx}, \cite{lovett-rank}.  Now the Theorem follows from Lemma \ref{rank-p-rank}.
\end{proof}

\begin{remark}\label{conj-bias} The dependence of $r$ on $s$ in (1) is polynomial, and double exponential in $d$. This was shown for  $d=2,3$ in \cite{s-h}, for  $d=4$ in  \cite{lampert} and in full generality in  \cite{Mi} (A similar bound but with a weak dependence in $|k|$ was proved independently in \cite{janzer}).
\end{remark}

\begin{remark}\label{gcs} By repeated applications of Cauchy-Schwartz  we have that 
\[
|\mE_{x \in V} e_q(P(x))| \le \|e_q(P)\|_{U_d}.
\]
\end{remark}

\begin{remark}\label{norm-bias-rank}
\leavevmode
\begin{enumerate}
\item  We use Claim \ref{pd} in our proof of Theorems \ref{A1} and \ref{B}.
\item\label{subspace-rank}
Let $P:V\to k$ be a polynomial of degree $d$ and rank $R$, and let $W\subset V$ be a subspace of codimension $s$. Then the rank of $P_{|W}$ is $\geq R-s.$ 
\begin{proof} We may assume that $\mV=\mA ^n$ and $\mW \subset \mV$ consists of vectors $(x_i),1\leq i\leq n$ such that $x_i=0$ for $i\leq s$. If the rank of $P_{|W}$ is $< R-s $ we can find polynomials $Q_j,R_j$,  with $ j\leq r$, fro some $r< R-s$ of degrees $<d$ such that $S_{|W} \equiv 0$ where  
$ S :=P-\sum _{j=1}^rQ_jR_j$. Since $S_{|W} \equiv 0$ there exists polynomials $T_i, 1\leq i\leq s$ of degrees $<d$ such that $S =\sum_{i=1}^sx_iT_i$.
\end{proof}
\end{enumerate}
\end{remark}

Now we can prove Proposition \ref{size}.
\begin{proof}
The number of points on $X_{\bar P}^{\bar b} = \{x: \bar P(x)=\bar b\}$ is given by 
\[
q^{-c}\sum_{\bar a \in k^c}\sum_{x \in V} e_q\left(\sum_{i=1}^c a_i(P_i(x)-b_i) \right).
\]
By Theorem \ref{bias-rank-1} and Remark \ref{gcs}, for any $s>0$ we can choose $r$ so that for any $\bar a \ne 0$ we have 
\[
\left|\sum_{x \in V} e_q\left(\sum_{i=1}^c a_i(P_i(x)-b_i)\right) \right| < q^{-s}|V|.
\]
\end{proof}

\subsection{Proof of Theorem \ref{A1}}

Let $\epsilon>0$ we say that a property holds for $\epsilon$ a.e.$s \in S$ if it holds for all but $(1-\epsilon)|S|$ of the elements in $S$. 
\begin{theorem}\label{need}
For any $\bar d$ there exists $C=C(\bar d)$ such that for any $s,m>0$ there exists an effective bound  $r=r(s,\bar d)$ such that if $\bar P=(P_i)_{ 1 \le i \le c}$ is a collection of polynomials  on $V=k^n$  with $\deg P_i =d_i$ and the nc-rank $\bar P$ is  $> r$ then:
\begin{enumerate}
\item  For any collection of polynomials $\bar R=(R_i)_{ 1 \le i \le c}$, with $R_i:k^m \to k$  of degree $d_i$,  there exist an affine map $w:k^m \to k^n$ such that 
$\bar P(w(x))=\bar R(x)$.  Furthermore, if we denote by $n_{\bar R}$ the number of such affine maps, then for any $\bar R_1, \bar R_2$ as above $|1-n_{\bar R_1}/n_{\bar R_2}| <q^{-s+m^C}$. 
\item  If $\bar P$ is homogeneous, then for any homogeneous collection $\bar R=(R_i)_{ 1 \le i \le c}$,  $R_i:k^m \to k$  of degree $d_i$,  there exist a linear map $w:k^m \to k^n$ such that $\bar P(w(x))=\bar R(x)$.  Furthermore, if we denote by $n_{\bar R}$ the number of such linear maps, then for any $\bar R_1, \bar R_2$ as above $|1-n_{\bar R_1}/n_{\bar R_2}| <q^{-s+m^C}$.
\end{enumerate}
\end{theorem}

\begin{proof}
We provide two proofs of the Theorem: one which is valid only in the case where char$(k)>d$, but is more algebraic in nature, and the other relies on multiple applications of Cauchy-Schwartz. \\ 

We start with the proof in the case where char$(k)>d$.
Since the proof for general $d$ involves many indices, we first prove the case when $c=1$ and $d=2$ so as to make the argument clear. 

We are given $P(t)=\sum_{1 \le  i \le j \le n}a_{ij}t_it_j +\sum_{1 \le  i \le n}a_{i}t_i +a$ of rank $r$.
Note that  for any linear form  $l(t)=\sum_{i=1}^n c_it_i$ we have that $P(t)+l(t)$ is of rank $\ge r$.

Denote  $w(x) = (w^1(x)+s^1, \ldots, w^n(x)+s^n)$, where the $w^i$ are linear forms. We can write
\[\begin{aligned}
P(w(x)) &= \sum_{1 \le i \le j \le n} a_{ij}(w^i(x)+s^i)(w^j(x)+s^j)+ \sum_{1 \le i \le n} a_{i}(w^i(x)+s^i)+a \\
&= \sum_{1 \le i \le j \le n} a_{ij}\sum_{k,l=1}^m w^i_{k}w^j_{l}x_kx_l+\sum_{1 \le i \le n} a_{i}\sum_{k=1}^m w^i_{k}x_k \\
&+  \sum_{1 \le i \le j \le n} \sum_{k=1}^m a_{ij}(s^iw^j_k+s^jw^i_k)x_k+ \sum_{1 \le i \le j \le n} a_{ij}s^is^j + \sum_{1 \le i \le n} a_{i}s^i +a,
\end{aligned}\]
which we can write as
\[\begin{aligned}
&\sum_{1 \le k < l \le m}   \sum_{1 \le i \le j \le n} a_{ij}(w^i_{k}w^j_{l} + w^i_{l}w^j_{k}) x_kx_l +
\sum_{1 \le  l \le m} \sum_{1 \le i \le j \le n} a_{ij}w^i_{l}w^j_{l}x^2_l \\
&  + \sum_{k=1}^m \left[\sum_{1 \le i \le n} a_{i}\ w^i_{k}
+ \sum_{1 \le i \le j \le n}  a_{ij}(s^iw^j_k+s^jw^i_k)\right]x_k + \sum_{1 \le i \le j \le n} a_{ij}s^is^j + \sum_{1 \le i \le n} a_{i}s^i +a.
\end{aligned}\]

Our aim is to show that the collection of coefficients for all monomials in the variables $x_j, \ 1 \le j \le m$, is of rank $\ge r$ (as polynomials in $w^i, s^i$):

\begin{claim}\label{ind} The collection below is of rank $\ge r-1$. 
\[\begin{aligned}
&\left\{ \sum_{1 \le i \le j \le n} a_{ij}(w^i_{k}w^j_{l} + w^i_{l}w^j_{k})\right\}_{1 \le k  < l \le m}
 \bigcup \   \left\{ \sum_{1 \le i \le j \le n}  a_{ij}(s^iw^j_k+s^jw^i_k)+ \sum_{1 \le i \le n} a_{i} w^i_{k}\right\}_{1 \le k \le m} \\
 & \bigcup   \left\{\sum_{1 \le i \le j \le n} a_{ij}s^is^j +\sum_{1 \le i \le n} a_{i} s^i \right\} \
 \bigcup \ 
\left\{ \sum_{1 \le i \le j \le n} a_{ij}w^i_{l}w^j_{l}\right\}_{1 \le l \le m}
\end{aligned}\]
\end{claim}

\begin{proof}
We need to show that any non-trivial linear combination
\[\begin{aligned}
&\sum_{1 \le k < l \le m}  b_{kl} \left[\sum_{1 \le i \le j \le n} a_{ij}(w^i_{k}w^j_{l} + w^i_{l}w^j_{k})\right]
+\sum_{1 \le l \le m}  b_{ll} \left[\sum_{1 \le i \le j \le n} a_{ij}w^i_{l}w^j_{l} \right]\\
&+ \sum_{1 \le k \le m} c_k \left[\sum_{1 \le i \le j \le n}  a_{ij}(s^iw^j_k+s^jw^i_k)+ \sum_{1 \le i \le n} a_{i} w^i_{k}\right]
+ d\left[\sum_{1 \le i \le j \le n} a_{ij}s^is^j +\sum_{1 \le i \le n} a_{i} s^i\right]
\end{aligned}\]
is of rank $\ge r$.  Suppose $b_{11} \ne 0$. Then we can write the above as
 \[
 b_{11}P(w_1) + l_{w_2, \ldots, w_m,s}(w_1)
\]
where $w_j=(w_j^1, \ldots, w_j^n)$, and $l_{w_2, \ldots, w_m,s}$ is affine in  $w_1$, so as a polynomial in $w_1$ this is of rank $\ge r$ and thus also of rank $\ge r$ as a polynomial in $w$.  Similarly in the case where $b_{ll} \ne 0$, for some $1 \le l \le m$.

Suppose $b_{12} \ne 0$.  We can write the above as
\[
(*) \quad  b_{12}Q(w_1, w_2) + l^1_{w_3, \ldots, w_m,s}(w_1) + l^2_{w_3, \ldots, w_m,s}(w_2)
\]
where $Q:V^2 \to k$, and $Q(t,t)=2P(t)$, and $l^i_{w_3, \ldots, w_m,s}:V\to k$ are affine maps. Thus restricted to the subspace
 $W$ where $w_1=w_2$ we get that $(*)$ is of rank $\ge r$ and thus of rank $ \ge r-1$ on $W$.  Similarly if $b_{kl} \ne 0$ for some
$k< l$. A similar analysis for the cases when $c_k$ or $d_k$ are not zero yields the desired result.  \\
\end{proof}

If $P$ is homogeneous  $P(t)=\sum_{1 \le  i \le j \le n}a_{ij}t_it_j$ of rank $r$, and $w: k^m \to k^n$ a linear map, we can write
\[\begin{aligned}
P(w(x)) &= \sum_{1 \le i \le j \le n} a_{ij}w^i(x)w^j(x) = \sum_{1 \le i \le j \le n} a_{ij}\sum_{k,l=1}^m w^i_{k}w^j_{l}x_kx_l
\end{aligned}\]
which we can write as
\[\begin{aligned}
&\sum_{1 \le k < l \le m}   \sum_{1 \le i \le j \le n} a_{ij}(w^i_{k}w^j_{l} + w^i_{l}w^j_{k}) x_kx_l +
\sum_{1 \le  l \le m} \sum_{1 \le i \le j \le n} a_{ij}w^i_{l}w^j_{l}x^2_l \\
\end{aligned}\]
By Claim \ref{ind} the collection of coefficients for all monomials in the variables $x_j, \ 1 \le j \le m,$ is also of  rank $\ge r-1$.  \\

For $d>2$ we perform a similar computation. To simplify the notation we carry it out in the case $P, R$ are homogeneous; the non-homogeneous case is similar. In this case it suffices to use linear maps. Denote $\mcI_n=\{I=(i_1, \ldots, i_d): 1 \le i_1 \le \ldots \le i_d \le n\}$, and for $t \in k^n$ denote by $t_I = \prod_{j=1}^dt_{i_j}$.
Let $P$ be a degree $d$ polynomial $P(t)=\sum_{I \in \mcI_n} a_{I} t_{I}$ on $k^n$  of rank $r$.\\

Let $w:k^m \to k^n$ be a linear map.  We can write
\[
P(w(x)) = \sum_{I \in \mcI_n} a_{I} \left(\prod_{j=1}^d w^{i_j}(x)\right)
=  \sum_{I \in \mcI_n}  a_{I}\sum_{l_1, \ldots, l_d=1}^m \left(\prod_{j=1}^d w^{i_j}_{l_j} \right) x_{l_1} \dots x_{l_d}.
\]
We rewrite this as
\[
\sum_{l_1, \ldots, l_d=1}^m \sum_{I \in \mcI_n}  a_{I} \left(\prod_{j=1}^d w^{i_j}_{l_j} \right) x_{l_1} \dots x_{l_d}
= \sum_{l \in \mcI_m } \left(\sum_{I \in \mcI_n}  a_{I} \sum_{\sigma \in S(l)} \left(\prod_{j=1}^d w^{i_j}_{\sigma_j} \right)\right)x_{l_1} \dots x_{l_d}\]
where $S(l)$ is the set of permutations of $l_1, \dots l_d$.  \\

Once again, our aim is to show that the collection of coefficients for all possible monomials in the variables $x_j, \ 1 \le j \le m$, is of rank $\ge r$:

\begin{claim} The collection below is of rank $\ge r-d+1$: 
\[
\left\{ \sum_{I \in \mcI_n}  a_{I} \sum_{\sigma \in S(l)} \prod_{j=1}^d w^{i_j}_{\sigma_j} \right\}_{l \in \mcI_m}
\]
\end{claim}

\begin{proof}
Consider a linear combination of the above collection with coefficient $c_l \ne 0$ for some $l \in \mcI_m$.  Consider $$Q(w_{l_1}, \ldots, w_{l_d})=  \sum_{I \in \mcI_n}  a_{I}  \sum_{\sigma \in S(l)} \prod_{j=1}^d w^{i_j}_{\sigma_j}.$$ 
Then  if $\Delta_wP(y) = P(y+w)-P(y)$ then $Q(w_{l_1}, \ldots, w_{l_d}) = \Delta_{w_{l_1}} \ldots  \Delta_{w_{l_d}}P(y)$.   Now we can write the given  linear combination as  $c_lQ(w_{l_1}, \ldots, w_{l_d})+T(w_{l_1}, \ldots, w_{l_d})$ where 
$T$ is a function of lower degree in $w_{l_1}, \ldots, w_{l_d}$ (it also depends on the $w_{t}$ for $t \ne l_i$, for $1 \le i \le d$). Since 
(char$(k), d)=1$, by Remark \ref{norm-bias-rank}(2) we have that the rank of $Q$ is the same as that of $P$,  and thus  the rank of the collection is the same as rank of $P$. 
\end{proof}

When  $c>1$,we are given $P_s(t)=\sum_{I \in \mcI_s}a^s_{I}t_I$, $1 \le s \le c$,  of rank $r$, where $\mcI_{d_s}(n)$ is the set of
ordered tuples $I=(i_1, \ldots, i_{d_s})$ with  $1 \le i_1 \le \ldots \le i_{d_s} \le n$, and $t_I= t_{i_1} \ldots t_{i_{d_s}}$.

Note that  for any polynomials $l_s(t)$ of degrees $<d_s$  we have that $\{P_s(t)+l_s(t)\}$ is also of rank $>r$.

We can write
\[
P_s(w(x)) = \sum_{I \in \mcI_{d_s}(n)} a^s_{I}w^I(x)
= \sum_{I \in \mcI_{d_s}(n)} a^s_{I}\sum_{l_1, \ldots, l_{d_s}=1}^m w^{i_1}_{l_1} \ldots w^{i_{d_s}}_{l_{d_s}}x_{l_1} \ldots x_{l_{d_s}},
\]
where $w^I = \prod_{i \in I} w^{i}$. 
For $( l_1, \ldots ,l_{d_s}) \in \mcI_{d_s}(m)$  the term  $x_{l_1} \ldots x_{l_{d_s}}$ has as coefficient
\[
\sum_{\sigma \in S_{d_s}} \sum_{I \in \mcI_s(n)} a^s_{I}w^{i_1}_{l_{\sigma(1)}} \ldots w^{i_{d_s}}_{l_{\sigma(d_s)}}.
\]
 We wish to show that  the collection
\[
\left \{\sum_{\sigma \in S_{d_s}} \sum_{I \in \mcI_{d_s}(n)} a^s_{I}w^{i_1}_{l_{\sigma(1)}} \ldots w^{i_{d_s}}_{l_{\sigma(d_s)}}
\right\}_{ 1 \le s \le c, ( l_1, \ldots ,l_{d_s}) \in \mcI_s(m)}
\]
is of rank $>r$.  Write $[1,c]=\bigcup_{f=2}^d C_f$ where $C_f=\{s: d_s=f\}$. 

We need to show that for any $f=2, \ldots, d$ if $B=(b^s_{ (l_1, \ldots, l_{d_s})})_{s \in C_f,  (l_1, \ldots, l_{d_s}) \in \mcI_{d_s}(m)}$ is not $\bar 0$, then
\[
\sum_{s\in C_f} \sum_{(l_1, \ldots, l_{f})} b^s_{ (l_1, \ldots, l_{f})}\sum_{\sigma \in S_{f}} \sum_{I \in \mcI_{f}(n)} a^s_{I}w^{i_1}_{l_{\sigma(1)}} \ldots w^{i_{f}}_{l_{\sigma(f)}}
\]
is of rank $>r$.  Suppose $(b^s_{ (l_1, \ldots, l_{f})})_{s\in C_f} \ne \bar 0$. Then restricted to the subspace
$w_{l_1} = \ldots = w_{l_{f}}$ we can write the above as
 \[
\sum_{s \in C_f} b^s_{ (l_1, \ldots, l_{f})}(f!) P_s(w_{l_1}) + R(w)
\]
where $w_j=(w_j^1, \ldots, w_j^n)$, and $R(w)$ is of lower degree in  $w_{l_1}$, so as a polynomial in $w_{l_1}$ this is of rank $>r$ and thus also of rank $>r$ as a polynomial in $w$. \\

Now (1), (2) follow from Proposition \ref{size}, since $| \mcI_m| = m^{C(\bar d)}$. \\
\ \\
We give an alternative proof that is valid for {\em all characteristic}. Following the computations above it suffices to show the following Claim: 

\begin{claim}\label{nc-version} For any $t>0$ there exists $r=r(t, \bar d)$ such that if the nc-rank of $\bar P$ is $>r$ then  for any polynomial $Q$ of degree $2\le f\le c$ that is  a non trivial combination of the polynomials in   the collection
\[
\left \{\sum_{\sigma \in S_f} \sum_{I \in \mcI_f(n)} a^s_{I}w^{i_1}_{l_{\sigma(1)}} \ldots w^{i_{f}}_{l_{\sigma(f)}}
\right\}_{  s \in C_f, ( l_1, \ldots ,l_{f}) \in \mcI_f(m)}
\]
we have that $| \mE e_q(Q)|<q^{-t}$.  
\end{claim}

\begin{proof}
Let $Q$ be of degree $f$, and write 
\[
Q=\sum_{(l_1, \ldots, l_{f})}\sum_{s\in C_f}   b^s_{ (l_1, \ldots, l_{f})}\sum_{\sigma \in S_f} \sum_{I \in \mcI_f(n)} a^s_{I}w^{i_1}_{l_{\sigma(1)}} \ldots w^{i_{f}}_{l_{\sigma(f)}}
\]
where $\bar b_{ (l_1, \ldots, l_{f})}= (b^s_{ (l_1, \ldots, l_{f})})_{s \in C_f} \ne 0$ for some $(l_1, \ldots, l_{f}) \in \mcI_f(m)$.
Observe that any $(l_1, \ldots, l_{f})$ determines a unique set of  variables $w_{l_1}, \ldots, w_{l_f}$, thus after $f$ applications of the Cauchy-Schwarz inequality
we can isolate this collection and arrive at the multilinear from associated with differentiating  
\[
\sum_{s\in C_f}   b^s_{ (l_1, \ldots, l_{f})}\sum_{\sigma \in S_f} \sum_{I \in \mcI_f(n)} a^s_{I}w^{i_1}_{l_{\sigma(1)}} \ldots w^{i_{f}}_{l_{\sigma(f)}}
\]
with respect to the associated set of variables. Namely we have 
\[
| e_q(Q)|_{U_1} \le \Big \| e_q\Big(\sum_{s\in C_f}   b^s_{ (l_1, \ldots, l_{f})}\sum_{\sigma \in S_f} \sum_{I \in \mcI_f(n)} a^s_{I}w^{i_1}_{l_{\sigma(1)}} \ldots w^{i_{f}}_{l_{\sigma(f)}}\Big) \Big \|_{U_f}
\]
But the latter is equal to 
\[
\mE_{w_{l_1}, \ldots, w_{ l_f}}e_q( \Delta_{w_{l_1}} \ldots  \Delta_{w_{l_d}} \bar b_{ (l_1, \ldots, l_{f})} \cdot \bar P^f(x)).
\]
where $\bar P^f = (P_s)_{s \in C_f}$. 
Now by Theorem \ref{bias-rank-1} we can choose $r$ such that if  $\bar P$ is of rank $>r$ then  the above is $<q^{-s}$. 
\end{proof}

This completes the proof of Theorem \ref{need}.
 \end{proof}


 \subsection{ Proof of  Theorem \ref{B}.}

In this subsection we prove  Theorem \ref{B}.

Let $V$ be a vector space and $l:V\to k$ be a non-constant linear function. For any subset $I$ of $k$ we denote $\mW _I=\{ v \in \mV |l(v)\in I\}$ so that  $\mW _b =\mW _{\{ b\} }$, for $b\in k$. For  $\mX \subset \mV$ we write $\mX _I=\mX \cap \mW_I$.

Theorem \ref{B} follows from the following proposition:

\begin{proposition}\label{line-plane} Fix $\bar d =\{ d_i\} $ and $m,s>0$.
Let $d:=\max _id_i $. There exists an effective bound $r=r(\bar d, s,m)$ such that for any finite field $k$,  any $k$-vector space $\mV$ and $ \bar P\in \mcP _{\bar d}(\mV)$ of nc-rank $>r$ the following holds. 
For any $b\in k$ and  $q^{-s}$-almost any affine $m$-dimensional subspace  $L \subset X_b$ there exists an $(m+1)$-dimensional affine subspace $M \subset X$ containing $L$ such that $M \cap X_0 \ne \emptyset$.
\end{proposition}

\begin{proof} 
We fix $d$ and define $d'= \min(d+1,q)$.
Let $M_0=\{ a_0,\dots ,a_{d}\}\subset  k$ be a subset of $d'$ distinct points.  To simplify notations we assume that $a_0=0$.

\begin{claim} Let $Q(x)$ be a polynomial of degree $\leq d$ such that
$Q_{|M_0}\equiv 0$. Then $Q(a)=0$ for all $a\in k$.
\end{claim}

\begin{proof}

Since any polynomial of degree $\leq d$ in one variable vanishing at $d+1$ point is equal to 0, the claim is true if $q\geq d+1$.
On the other hand if $d\geq q$ then there is nothing to prove.
\end{proof}

Let $J(d)$ be the subset  of $[0,d]^{m+1}$ of tuples $t=(t_1, \ldots, t_{m+1} )$ such that $0 \le t_{m+1} \le  \ldots \le t_1$. Let 
$T^{m+1}:=\{ a_{t} = (a_{t_1}, \ldots, a_{t_{m+1}}): t \in J(d)\}\subset k^{m+1}$.

\begin{claim}\label{reduce} Let $Q(x_1, \ldots, x_{m+1})$ be a polynomial of degree $\leq d$ such that
$Q_{|T ^{m+1}}\equiv 0$. Then $Q=0$.
\end{claim}

\begin{proof} Our proof is by induction in $m$ and for a fixed $m$ by induction in $d$. 

Consider first the case when 
$m=1$. We will write $x,y$ instead of $x_1,x_2$. We have $T^2=\{ (a_{t_1},a_{t_2}): 0\leq t_2\le t_1 \leq d\}$. 

We prove the claim by induction in $d$. Let
$Q=\sum _{a,b}q_{a,b} x^ay^b$, with $a+b\leq d$. The restriction of $Q$ to the line $\{ y=0\}$ is equal to
$Q^0(x)=\sum _{a\leq d}q_{a,0}x^a$. Since $Q^0_{|T^2}\equiv 0$ we see that $Q^0=0$. So $Q(x,y)=yQ'(x,y)$. So, by the inductive assumption, we have $Q'\equiv 0$.

Assume now the validity of Claim for polynomials in $m$ variables and for polynomials in $m+1$ variables degree $\leq d-1$.
Let  $Q(x_1, \ldots, x_{m+1})$ be a polynomial of degree $\leq d$ such that
$Q_{|T^{m+1}}\equiv 0$. Let $R$ be the restriction of $Q$ on 
the subspace of points $(x_1,\dots ,x_{m+1})$ such that $ x_{m+1} =0$. Since 
$R_{|T^m} \equiv 0 $ it follows from the inductive assumption that $R \equiv 0$ and therefore $Q=Q'x_{m+1}$ where $\operatorname{deg}(Q')=d-1$. Since
$Q'_{|T}\equiv 0$ we see from the inductive assumption that $Q'\equiv 0$.
\end{proof}

Denote $I(d)$ the set of indexes
\[
I(d)= \left\{ t:=(t_1, \ldots, t_{m+1}) \in J(d): 1 \le t_{m+1} \right\}.
\]
An affine $m$-dimensional subspace  in $X_b$ is parametrized as $\{x+\sum_{i=1}^m s_iy_i; \ s_i \in k\}$, such that for all $2 \le e \le d$, all  $P^e_j \in \bar P^e$ we have
\[
(*) \quad  P^e_j\left(x+\sum_{i=1}^m s_iy_i\right)=0,\  l(x)=b, \ l(y_i)=0.
\]
Let $Y$ be the set of $(x,\bar y)$ satisfying $(*)$.

We need to show that almost every $(x, \bar y) \in Y$ we can find $z$ such that for all $2 \le e \le d$, all  $P^e_j \in \bar P^e$ we have
\[
\forall s,s_1, \ldots, s_m \in k, \ \quad P^e_j\left(x+\sum_{i=1}^m s_iy_i+sz\right)=0, \  l(z)=-b,  
\]
or  alternatively 
\[
\forall s,s_1, \ldots, s_m \in k,  \quad P^e_j\left(x+\sum_{i=1}^m s_iy_i+sz\right)=0, \  l\left(x+\sum_{i=1}^m s_iy_i+sz\right)=(1-s)b. \]
By Claim \ref{reduce}, since $P^e_j$ is of degree $e$,  we can reduce this system to
\[
P^e_j\left(x+\sum_{i=1}^m a_{t_i}y_i+a_{m+1}z\right)=0, \    l\left(x+\sum_{i=1}^m a_{t_i}y_i+a_{t_{m+1}}z\right)=(1-a_{t_{m+1}})b, \quad t \in I(e).
\]
Denote $I = \sum_{e \in [d]} |I(e)|$.
Fix $(x,\bar y) \in Y$ and estimate the number of solutions $A(x ,\bar y)$ to the above system of equations, which is given by
\[\begin{aligned}
& q^{-2I} \sum_{z} \sum_{e \in [d], \bar c^e_{t^e}: t^e \in I(e), h_t : t \in I(d)}\\ &e_q\Big( \sum_{e \in [d]}\sum_{t^e} \bar c^e_{ t^e} \cdot \bar P^e \Big(x+\sum_{i=1}^m a_{t^e_i}y_i+a_{t^e_{m+1}}z\Big) 
+ \sum_t h_{t}\Big( l  \Big(x+\sum_{i=1}^m a_{t_i}y_i+a_{t_{m+1}}z \Big)+(a_{t_{m+1}}-1)b  \Big)\Big).
\end{aligned}\]
Suppose  $\bar c_{t^e}^e= 0$ for all $ t^e$, but $\bar h \ne 0$, and recall that $l(x)=b$, $l(y_i)=0$. We get
\[\begin{aligned}
&\sum_ze_q\left(\sum_{t} h_{t} \left(l\left(x+\sum_{i=1}^m a_{t_i}y_i+a_{t_{m+1}}z\right)+(a_{t_{m+1}}-1)b\right)\right) \\
&= \sum_z e_q\left(\sum_{t} h_{t} (a_{t_{m+1}}l(z)+a_{t_{m+1}}b)\right).
\end{aligned}\]
Now if $\sum_{t} h_{t} a_{t_{m+1}}l(z) \not \equiv 0$ then the sum is $0$. Otherwise also $\sum_{t} h_{t} a_{t_{m+1}}b=0$ so that the sum is $|V|$.
\ \\
Now suppose $\bar c_{t_0}^e \ne 0$ for some $t_0$, and let $e$ be the largest degree for which this holds. Let
\[
T(x,y,z)= \sum_{t^e} \bar c^e_{ t^e} \cdot \bar P^e \Big(x+\sum_{i=1}^m a_{t^e_i}y_i+a_{t^e_{m+1}}z\Big) +Q(x,y,z)
\]
where $Q$ is of degree $<e$. 
We estimate
\[
B_{t_0}=\mE_{x,\bar y \in V}\left|\mE_ze_q(T(x,y,z))\right|^2.
\] 

\begin{lemma}\label{complexity} For any functions $f_t : V \to \mC$, $\|f_t\|_{\infty} \le 1$,  $t  \in I(d)$, we have 
\[
\left|\mE_{x,\bar y,z,z'}\prod _{t \in I(d)} f_{t}\big(x+\sum_{i=1}^m a_{t_i}y_i+a_{t_{m+1}}z\big) \bar f_{t}\big(x+\sum_{i=1}^m a_{t_i}y_i+a_{t_{m+1}}z+ a_{t_{m+1}}z'\big)\right| \le \|f_{t_0}\|_{U_{d}}.
\]
\end{lemma}

\begin{proof}
To simplify the notation we prove this in the case $m=1$.  
Without loss of generality $a_1=1$ (make a change of variables $y \to a_1^{-1}y,z \to a_1^{-1}z )$.
We prove this by induction on $d$. When $d=1$, $I(d)=\{(1,1)\}$, and the claim in this case follow from the following inequality
\[
\big|\mE_{x,y,z,z'}f_1(x+y+z)f_2(x+y+z+z')\big| = \big|\mE_{z,z'}f_1(z)f_2(z')\big| \le  \big|\mE_{z}f_i(z)\big| \le  \|f_{t_0}\|_{U_1}.
\]\
Assume $d>1$.   We can write the average as
\[\begin{aligned}
&\mE_{x,y,z,z'} \prod_{(i,j) \in I(d-1)} f_{i,j}(x+a_iy+a_jz) \bar f_{i,j}(x+a_iy+a_jz+a_jz')\\
& \qquad  \qquad \prod_{1\le j \le d} f_{d,j}(x+a_dy+a_jz) \bar f_{d,j}(x+a_dy+a_jz+a_jz'). \\
\end{aligned}\]
Shifting $x$ by $a_dy$ we get
\[\begin{aligned}
&\mE_{x,y,z,z'} \prod_{(i,j) \in I(d-1)} f_{i,j}(x+(a_i-a_d)y+a_jz) \bar f_{i,j}(x+(a_i-a_d)y+a_jz+a_jz')\\
& \qquad  \qquad \prod_{1\le j \le d} f_{d,j}(x+a_jz) \bar f_{d,j}(x+a_jz+a_jz').  \\
\end{aligned}\]
Applying the Cauchy-Schwarz inequality we can bound the above as
\[\begin{aligned}
&\big[\mE_{x,y,y',z,z'} \prod_{(i,j) \in I(d-1)} f_{i,j}(x+(a_i-a_d)y+a_jz) \bar f_{i,j}(x+(a_i-a_d)y+a_jz+a_jz')\\
& \qquad \qquad \prod_{(i,j) \in I(d-1)}\bar f_{i,j}(x+(a_i-a_d)y+(a_i-a_d)y'+a_jz)  \\
& \qquad \qquad \qquad \qquad f_{i,j}(x+(a_i-a_d)y+(a_i-a_d)y'+a_jz+a_jz')\big]^{1/2}.\\
\end{aligned}\]
Shifting $x$ by $a_dy$ and rearranging we get
\[\begin{aligned}
&\big[\mE_{x,y,y',z,z'} \prod_{(i,j) \in I(d-1)} f_{i,j}(x+a_iy+a_jz) \bar  f_{i,j}(x+a_iy+(a_i-a_d)y'+a_jz)\\
&\prod_{(i,j) \in I(d-1)}\bar f_{i,j}(x+a_iy+a_jz+a_jz')  f_{i,j}(x+a_iy+(a_i-a_d)y'+a_jz+a_jz')\big]^{1/2}.\\
\end{aligned}\]

Now if we denote
\[
g_{i,j, y'}(x) =f_{i,j}(x) \bar f_{i,j}(x+(a_i-a_d)y'),
\]
then by the induction hypothesis we get that the above is bounded by
\[
\big[\mE_{y'}\|g_{i,j, y'}\|_{U_{d-1}}\big]^{1/2} \le \|f_{i,j}(x)\|_{U_d}
\]
for any $(i,j) \in I(d-1)$. \\

We do a similar computation for $(i.j) \in I(d)\setminus \{I(d-1), (d,1)\}$ , splitting
\[\begin{aligned}
&\mE_{x,y,z,z'}\prod_{(i,j) \in I(d-1)} f_{i+1,j+1}(x+a_{i+1}y+a_{j+1}z) \bar f_{i+1,j+1}(x+a_{i+1}y+a_{j+1}z+a_{j+1}z')\\
& \qquad  \qquad \prod_{1\le j \le d} f_{j,1}(x+a_jy+z) \bar f_{j,1}(x+a_jy+z+z'), \\
\end{aligned}\]
and shifting  $x$ by $z$ to get
\[\begin{aligned}
&\mE_{x,y,z,z'}\prod_{(i,j) \in I(d-1)} f_{i+1,j+1}(x-z+a_{i+1}y+a_{j+1}z) \bar f_{i+1,j+1}(x-z+a_{i+1}y+a_{j+1}z+a_{j+1}z')\\
& \qquad  \qquad \prod_{1\le j \le d} f_{j,1}(x+a_jy) \bar f_{j,1}(x+a_jy+z').  \\
\end{aligned}\]

The only term left uncovered is $f_{d,1}$,  so we split
\[\begin{aligned}
&\mE_{x,y,z,z'} \prod_{(i,j) \in I(d-1)} f_{i+1,j}(x+a_{i+1}y+a_{j}z) \bar f_{i+1,j}(x+a_{i+1}y+a_{j}z+a_{j}z')\\
& \qquad  \qquad \prod_{1\le i \le d} f_{i,i}(x+a_{i}y+a_iz) \bar f_{i, i}(x+a_{i}y+a_iz+a_iz'). \\
\end{aligned}\]
 We make the change of variable $z \to  z-y$ to get
 \[\begin{aligned}
&\mE_{x,y,z,z'} \prod_{(i,j) \in I(d-1)} f_{i+1,j}(x+a_{i+1}y+a_{j}(z-y)) \bar f_{i+1,j}(x+a_{i+1}y+a_{j}(z-y)+a_{j}z')\\
& \qquad  \qquad \prod_{1\le i \le d} f_{i,i}(x+a_iz) \bar f_{i, i}(x+a_iz+a_iz').  \\
\end{aligned}\]
Observe that the argument of $f_{d,1}$ is  $(x+(a_{d}-a_1)y+a_{1}z)$, so that the coefficient of $y$ is not zero.  Now proceed as in previous cases. 
 \end{proof}

By the Lemma \ref{complexity} we obtain that $B_{t_0}$ is bounded by $\|e_q( \bar c^e_{ t_0^e} \cdot \bar P^e)\|_{U_e}$.
By Theorem \ref{bias-rank-1}  there exists an effective bound  $r=r(s, d)$, such that if  $P$ is of rank $>r$ then  $\|e_q( \bar c^e_{ t_0^e} \cdot \bar P^e)\|_{U_e}<q^{-s}$.  It follows that we can choose $r$ so that for $q^{-s}$-almost all $(x,y) \in Y$,  the number of solutions $A(x, \bar y)$ is bounded below by $|V|q^{-4I}$.
\end{proof}

\subsection{Proof of Theorems \ref{AC} and \ref{BC} }
We fix $m,d$ and $v$.
 As follows from Theorem \ref{A1}, there exists an effective bound $r=r(m,d,c)$ such that for any finite field $\mF _q$, any $\mF _q$-vector space $\mV$ the map $\ti \kk _P(\mF _q)$ is surjective for any family $\bar P=(P_i)$ of polynomials $P_i\in \mcP _d(\mV)$ such that  $r_{nc}(\bar P)\geq r$.

We now consider the case when $k$ is an algebraically closed field.

\subsubsection{The surjectivity of $\ti \kk _P(k)$}

We fix  $\mV =\mA ^n$ and consider $\mcP _d(\mV)^c$ as a scheme defined over $\mZ$. Let $T$ be the set of sequences $(a_i,b_i),1\leq i\leq r$, such that $0\leq a_i,b_i <d$ and $a_i+b_i \leq d$. For any $t=\{(a_i,b_i)\} \in T$ we denote by $\nu _t:\oplus _{i=1}^r \mcP _{a_i}(\mV) \otimes \mcP _{b_i}(\mV) \to  \mcP _d(\mV) $ the linear map given by 
$$\nu _i(\{ Q_i\otimes R_i\})=\sum _{i=1}^r Q_i R_i.$$ 
Let $\mZ$ be the union of images of the maps of $\nu _t,t\in T$. 

Let $\mY \subset \mcP _d(\mV)^c$ be constructible subset of families
 $\bar P=\{ P_i\}$ such that $\mZ \cap L_{\bar P}=\{0\}$ where
 $L_{\bar P} \subset \mcP _d(\mV) $ is the span of $( P_i)$.

So $\mY (k)\subset \mcP _d(\mV)(k)$ consists families $\bar P$ of polynomials of rank $>r$. We define $\mR \subset \mY$ as the constructible subset of $\bar P\in \mY$ such that the morphism $\ti \kk _{\bar P}$ is not surjective. Our goal is to show that $\mR =\emp$. 

We first consider the case when $k= \bar \mF _p $ is the algebraic closure of $\mF _p$.

\begin{claim}\label{p} $\mR (\bar \mF _p)=\emp$.
\end{claim}

\begin{proof} Assume that
 $\mR (\bar \mF _p)\neq \emp$.  Then there exists $\bar P\in \mY (\bar \mF _p)$ such that the map $\ti \kk _{\bar P}(\bar \mF _p)$ is not surjective. Then there exists 
 $Q\in \mcP _{\bar d}(\bar \mF _p)$ which is not in the image of $\ti \kk _ {\bar P}(\bar \mF _p).$ 
By definition there exists $l\geq 1$ such that  $\bar P\in \mY (\bar \mF _q)$ and 
 $Q\in \mcP _{\bar d}( \mF _q),q=p^l$. 
But as follows from  Theorem \ref{A1} there exists  $\phi \in \text{Aff}_m(\mF _q)$ such that $Q=\ti \kk _ {\bar P}(\phi)$. But the existence of such affine map $\phi$ contradicts the assumption that $Q$ is not in the image of $\ti \kk _ {\bar P}(\bar \mF _p).$
\end{proof}

\begin{corollary}\label{kappa}\leavevmode
\begin{enumerate}
\item The map $\ti \kk _P(k)$ is surjective for  any algebraically closed field $k$ and a polynomial $P\in \mcP _d(\mV)$ of nc-rank $>r(m,d)$.
\item The map $\ti \kk _P(k)$ is surjective for 
any algebraically closed field $k$ of characteristic $0$ and a polynomial $P\in \mcP _d(\mV)$ of nc-rank $>r(m,d)$.
\end{enumerate}
\end{corollary}

\begin{proof} The part $(1)$ follows from the completeness of the theory $ACF_p$ of algebraically closed fields of a fixed characteristic $p$.  

To prove the part $(2)$ one choses a non-trivial ultrafilter $\mcU$ on the set of primes and 
considers the $\mcU$-ultraproduct of theories $ACF_p$. Let $l$ be the $\mcU$-ultraproduct of fields $\bar \mF _p$. 

Let ACF be the theory of algebraically closed fields and 
$\alpha$ be the formula in ACF expressing the surjectivity of the map $\kk _P$.
As follows from Claim \ref{p}  $\alpha$ holds for algebraic closures of fields $\mF _p$. We fix now $m,d,c$ and for any $n\geq 1$ define by $\alpha _n$ the following formula in $ACF$.

For any family $\bar P=(P_i), P_i\in k[x_1.\dots ,x_n]$ such that $r_{nc}(\bar P)\geq r(m,d,c)$ the map $\kk _{\bar P}:\text{Aff} _m(\mA ^n)\to (\mcP _d(\mA ^m))^c$ is surjective.

By the Theorem of {\L}o\'s applied to the formula $\alpha _n$
we see that the  map $\kk _P(l)$ is surjective for any family $\bar P\in \mcP _d(\mV)$, $\dim(\mV)=n$ of nc-rank $>r$. Since the theory $ACF_p$ of algebraically closed fields of  characteristic $0$ is complete and $n\geq 1$ is arbitrary the corollary is proved.  
\end{proof}

\subsubsection{The computation of the dimensions of fibers of $\ti \kk (k)$}

Let $\Hom_{\af}(\mW ,\mV)$ be the variety of affine maps from $\mW$ to $\mV$ and 
let $\mT \subset \mY$ be be the subscheme of polynomials $P$ such that there exists $Q\in \mcP _d(\mA ^m)$ such that 
 $\dim (\kk _P^{-1}(Q))\neq \dim (\Hom_{\af}(\mW ,\mV))-\dim(\mcP _d(\mW))$.

We want to prove that $\mT =\emp$. 
The same arguments as before show that 
it is sufficient to prove that  
$$\dim (\kk _P^{-1}(Q))= \dim (\Hom_{\af}(\mW ,\mV))-\dim(\mcP _d(\mW))$$ for all finite fields $k=\mF _q$ and $Q\in \mcP _d(\mA ^m)(k)$.
Let $w:= \dim (\Hom_{\af}(\mW ,\mV))-\dim(\mcP _d(\mW))$. As 
follows from \cite{LW} that there exists constant $A(n,d),l>0$ such that 
$||\kk _P^{-1}(Q)(\mF _{q^m})|-q^w|\leq A(n,d) q^{w-1/2}$ for any $q=p^{lm}$.

This implies that 

$$\dim (\kk _P^{-1}(Q))=\lim _{m\to \infty}
\frac {\log _q (| \kk _P^{-1}(Q)(k_m) |)}{lm},$$ 
where $k_m=\mF _{q^{ml}}$ is the extension of degree $l$. Now the equality $$\dim (\kk _P^{-1}(Q))=w$$ follows from 
Theorem \ref{A1}. 

Part (3) of Theorem \ref{AC} follows now from Theorem 23.1 in \cite{M}, and part (4) follows  from the part (3).

The derivation of Theorem \ref{BC} from Theorem \ref{B} is completely analogous.

\subsection{Proof of Theorem \ref{Jan} } 
{\em Proof of Theorem \ref{Jan}.} Let $\mcG$ be a subfunctor of $\mcF _d$ such that $r(P), P\in \mcG (W) $ is not bounded above. We want to show that $\mcG (W)=\mcF _d(W)$ for any finite-dimensional $k$-vector space $W$.

Let $m=\dim(W)$ and choose a polynomial $P\in \mcG (V)$, where $V$ is a $k$-vector space $V$ such that $r_{nc}(P)\geq r(m,d)$, where  $r(m,d)$ is as in the Corollary \ref{kappa}. Then for any polynomial $Q$ on $W$ of degree $d$, there exist an affine map $\phi :W\to V$ such that $Q=\phi^\star (P)$. We see that $\mcG (W)=\mcF _d(W)$.
\qed

\section{Extending weakly polynomial functions from high rank varieties}

\subsection{Introduction} 
We fix $d,a,c\geq 1$  and a field $k$   such
that $|k|>ad$ and that there exists a root of unity $\beta \in k$ of order $m>ad$. A field is  {\em admissible}
if it satisfies these conditions. 

\begin{definition}\label{weak-def-1}
Let  $V$ be a $k$-vector space,  and  let $X\subset V$. We say that a function $f:X \to k$ is {\it weakly polynomial} of degree $\leq a$ if  restrictions $f_{|L}$ to  affine subspaces  $L \subset X$ are polynomials of degree $\leq a$.  
\end{definition}

\begin{remark} If $|k|>a$ it suffices to check this on $2$-dimensional subspaces (see \cite{KR}, Theorem 1).
Namely a function is {\it weakly polynomial} of degree $\leq a$ if the restriction $f_{|L}$ to $2$-dimensional affine subspace  $L \subset X$ is a polynomial of degree $\leq a$.
\end{remark}

The goal of this section is to show in the case when $k$ is admissible field which is either finite or algebraically closed then any weakly polynomial function $f$ on a subvariety  $X \subset V$ of a sufficiently high rank extends to a polynomial $F$ of degree $\leq a$ on $V$. The main difficulty is in the case when  $a\geq d$ when an extension $F$ of $f$ is not unique. 

To state our result properly we introduce the following definition:

\begin{definition} An algebraic $k$-subvariety  $\mX \subset \mV$ satisfies the condition $\star ^k_{a}$
if any weakly polynomial function of degree $\leq a$ on $X$ is a restriction  of a polynomial function of degree $\leq   a$ on $V$.
\end{definition}

The following example demonstrates the  existence of cubic 
surfaces $\mX \subset \mA^2$  which do not have the property 
$\star^k_{1}$ for any field $k \neq \mF _2 $.

\begin{example} Let  $V=k^2$, $Q=xy(x-y)$. Then $X=X_0\cup X_1\cup X_2$ where 
$X_0=\{ v\in V|x=0\}, X_1=\{ v\in V|y=0\} , X_2=\{ v\in V|x=y\} $.
The function $f:X\to k$ such that $f(x,0)=f(0,y)=0,f(x,x)=x$ is weakly linear but one can not extend it to a linear function on $V$.
\end{example}

The main result of this section is that high rank hypersurfaces over admissible fields satisfy   $\star^k_a$. 

\begin{theorem}\label{main} 
There exists an $r=r(a,d)$ such that for any admissible field $k$ 
which is either finite or algebraically closed, 
any hypersurface $\mX$ of degree $d$ and nc-rank $\geq r$ in a $k$-vector space satisfies  $\star^k_a$.
\end{theorem}

The result extends without difficulty to complete intersections $\mX \subset \mV$ of bounded degree and codimension, and high rank (see Definition \ref{rank}).
\begin{theorem}\label{main1} 
For any  $c>0$, there exists an effective bound $r=r(a,d,c)$  such that for any admissible  field $k$, which is either finite or algebraically closed, a
$k$-vector space $\mV$, any  subvariety  $\mX \subset \mV$ of codimension $c$, degree $d$ and nc-rank $\geq r$ satisfies  $\star ^k_a$. 
\end{theorem}

Our proof of Theorem \ref{main1} consists of two steps. 

We first construct for any $d$ hypersurfaces  $\mX _n\subset \mV _n$ over $\mZ$ of degree $d$ and 
arbitrary high rank such that for any admissible field $k$ and any $c$ the subset $\mX_n (k)^c\subset \mV_n (k)^c$
satisfying the conditions of Theorem \ref{main1}. This result is purely algebraic. In the second step we show how to derive the general case of Theorem \ref{main1} from this special case.

\begin{remark} 
 The case $a <d$ was studied in \cite{kz-uniform}. The case $a=d=2$ of was studied in \cite{kz}, and a bilinear version of it was studied in \cite{GM}, where it was applied as part of a  quantitative proof for the inverse theorem for the $U_4$-norms over finite fields. We expect the results in this paper to have similar applications to a quantitative proof for the inverse theorem for the higher Gowers uniformity norms, for which at the moment only a non quantitative proof using ergodic theoretic methods exists \cite{btz, tz, tz-1}. 
\end{remark}

\subsection{Construction of an  explicit collection of subvarieties}

Let $\mW:=\mA ^d, \mV_n:=\mW ^n$, and 
$P_n:\mV_n \to \mA$ be given by $P_n(w_1,\dots ,w_n)= \sum _{i=1}^n \mu(w_i)$, where  $\mu:\mW \to \mA$  is the product $\mu(x^1, \dots ,x^d):= \prod _{j=1}^dx^j$. Let  $\mX _n\subset \mV _n$ be the hypersurface defined by the equation $P_n(v)=0$.

\begin{theorem}\label{const} \leavevmode
\begin{enumerate}\item There exists $\epsilon>0$ such that the nc-rank $r_{nc}(P_n)\geq n^{\epsilon}$.
\item 
For any admissible field $k$ and any $c\geq1$ the subvariety $(\mX _n) ^c \subset \mV ^c$ has the property  $\star _a$.
\end{enumerate}
\end{theorem}
\begin{remark}To simplify notations we present the proof only in the case when $c=1$. The proof in the general case is completely analogous.
\end{remark}


\subsection{Proof of Theorem \ref{const}} 
\subsubsection{Proof of the part (1) of Theorem \ref{const}} In this subsection we prove the part $(1)$ of 
Theorem \ref{const}.

\begin{proof}
First we note that for a non trivial character $\psi$ on $k$ we have 
\[
|\mE_{w \in V} \psi (\mu(w))|  = t <1. 
\]
Now we observe that
\[
\tilde \mu (u^1_1, \ldots,u^1_d, \ldots, u^1_1, \ldots,u^1_d)|_{\{u^l_j=0, l \ne j\}} = u^1_1u^2_2\cdots u^d_d
\]
Denote 
\[
U=\{((u_1)_1, \ldots, (u_n)_1,\ldots,  (u_1)_d, \ldots,  (u_n)_d) \in V^d: (u_i)^l_j=0, l \ne j, i \in [n]\}
\]
Then restricted to $U$ we have 
\[ 
\ti P_n((u_1)_1, \ldots, (u_n)_1,\ldots,  (u_1)_d, \ldots,  (u_n)_d)|_U= \sum_{i=1}^n \mu((u_i)^1_1, (u_i)^2_2, \ldots , (u_i)^d_d),
\]
so that 
\[
\mE_{u \in V^d} \psi( \ti P_n(u)) \le \mE_{u \in U} \psi( \ti P_n(u)) = |\mE_{w \in W} \psi (\mu(w))|^{n}  = t^{n}.
\]
It follows by \ref{bias-rank-1} that $ \ti P_n$ is of rank $> n^{\epsilon}$ for some $\epsilon >0$. 
\end{proof}

\begin{definition}\label{X}
\begin{enumerate} 
\item For any set $X$ we denote by  $k[X]$  the space of $k$-valued functions on $X$.
\item For a subset $X$ of a vector space $V$, we denote by $\mcP _a^w(X) \subset k[X] $ the subspace of 
 weakly polynomial functions of degree $\leq a$.
\item We denote by $\mcP _a(X) \subset \mcP _a^w(X) $ the subspace of functions $f:X\to k$ which are restrictions of polynomial functions on $V$ of degree $\leq a$.
\item  As before we define $\mW=\mA^d, \mV _n:=\mW ^n $ and denote by  $\mu$  the product map $\mu:\mW \to \mA$ given by 
\[
\mu (a^1,\dots ,a^d)= \prod _{s=1}^d a^s.
\]
We write elements of 
$V_n$ in the form 
$$v= (w_1,\dots ,w_n), \ 1\leq i\leq n,\ w_i\in W.$$
\end{enumerate}
\end{definition}
It is clear that Theorem \ref{main1} is equivalent to the following statement.
\begin{theorem}\label{equality} Let $k$ be an admissible field, then $\mcP _a^w(X_n) = \mcP _a(X_n) $.
\end{theorem}

We fix $n$ and write $\mX$ instead of $\mX _n$, and $\mV$ instead of $\mV _n$.
The proof of the part (2) of Theorem \ref{const}
uses the  existence of a large group  of symmetries of $X$, the existence of a linear subspace $L\subset V$ of large dimension and the existence of the subroup $\D \subset k^\star , \D \cong \mZ /m\mZ$ for $m>ad$.

\begin{proof}We start the proof of Theorem \ref{equality} 
with the following result.
\begin{claim}\label{many}
Let $Q$ be a polynomial of degree $\leq ad$ on $k^N$ such that $Q_{|\D ^N}\equiv 0.$ Then $Q=0$.
\end{claim}
\begin{proof} The proof is by induction on $N$. If $N=1$   then $Q=Q(x)$ is polynomial such that $Q(\delta) =0$ for $\delta \in \D$. Since $|\D| > ad$ we see that $Q=0$.

Assume that the result is known for $N'=N-1$. Let $Q$ be a polynomial of degree $\leq ad$ on $k^N$ such that $Q_{|\D ^N}\equiv 0.$ By induction we see that $Q(\delta ,x_2,\dots ,x_s)\equiv 0$ for all $\delta \in \D$.
Then for any $x_2,\dots ,x_s $ the polynomial $x\to Q(x,x_2,\dots ,x_s)$ vanishes for all $\delta \in \D$. Therefore $Q(x,x_2,\dots ,x_s) =0$ for all $x\in k$.
\end{proof}

\begin{definition}\label{ga}\leavevmode
\begin{enumerate}
\item $\GG :=(S_d)^n$. The group $\GG$ acts naturally on $X$.
\item    $L:=\{(c_1, \ldots, c_n)\in k^n |\sum_{i=1}^nc_i=0\}.$ 
\item $L_\D=(\D) ^n\cap L\subset k^n$.
\item For $c\in k$ we write $w(c):=(c,1,\dots ,1)\in W$.
\item   $\kk :L\ho X\subset V$ is  the linear map given by 
$$\kk (c_1, \ldots, c_n) := (w(c_1), \ldots, w(c_n))$$ and write 
$\kk _\Gg :=\Gg \circ \kk ,\Gg \in \GG$.
\item For any function $f:X\to k, \Gg \in \GG$ define a function $h_{\Gg ,f} :L\to  k$ by 
$h_{\Gg ,f} :=f\circ  \kk _\Gg$.
\item $T_1:= \{ (u_1,\dots ,u_d)\in (\D)^d|\prod _{j=1}^du _j=1\}$ and $T:=T_1^n$.
\item We denote by $\zeta _i:T_1\ho T,1\leq i\leq n$ the imbedding onto the i-th component.
\item For any $j,j'$, $1\leq j\neq j'\leq d$ we denote by $\phi _{j,j'}: \D \to T_1 $ the 
morphism such that $\phi _{j,j'}(u)= ( x_l(u), 1\leq l\leq d ) $ where $x_j(u)=u, x_{j'}(u)=u^{-1} $ 
and $x_l(u) =1$ for $l\neq j,j'$.
\item We denote by $\Theta _1$ the group of homomorphisms  $\chi :T_1\to k^\star$.
\item $ \Theta =( \Theta _1)^n$.
\item For $\chi \in \Theta _1 , j,j', 1\leq j\neq j'\leq d $ we define
a homomorphism $\chi _{j,j'}: \D \to k^\star $ by $\chi _{j,j'}:=\chi \circ \phi _{j,j'} $. Since $\D \cong \mZ /m\mZ$ there exists unique $\alpha _{j,j'}(\chi )\in (-m/2,m/2]$ such that 
$\chi _{j,j'}(u)= u^{\alpha _{j,j'} (\chi )}$ for any $u\in \D$. 
\item  $\Theta _1^{adm}:=\{ \chi  \in \Theta _1 : | \alpha _{j,j'}(\chi ) |\leq a\}$.
\item  $\Theta_1^{adm,+}:= \{ \chi \in \Theta_1^{adm} : \alpha _{j,j'} (\chi)\geq 0, j<j' \}$.
\item Let   $\Theta^{adm,+}:= (\Theta _1 ^{adm,+})^n$ and 
$\Theta^{adm}:=(\Theta _1 ^{adm})^n$.
\item For any $k$-vector space $R$, a representation $\pi : T \to \Aut (R)$ and  $\theta \in \Theta $ we define 
$$R^\theta =\{ r\in R|\pi(t)r=\theta (t)r, \ t\in T\}.$$
\end{enumerate}
\end{definition}
\begin{remark} Since $|T|$ is prime to $q:=\operatorname{char}(k)$ the Maschke's theorem implies the direct sum decomposition 
$R=\oplus _{\theta \in \Theta }R^\theta$.
\end{remark}

\begin{claim}\label{pol} For any $f\in \mcP _a^w(X) , \Gg \in \GG $ the function 
$h_{\Gg ,f} $ is a polynomial of degree $\leq a$.
\end{claim}

\begin{proof} Since $f\in \mcP _a^w(X) $ we have  $h_{\Gg ,f}\in \mcP _a^w(L) $. Since $L$ is linear space we see that $h_{\Gg ,f} $ is a polynomial of degree $\leq a$.
\end{proof}

\begin{claim}\label{plus}\leavevmode
\begin{enumerate}
\item The subset $\Theta^{adm}$ of $\Theta$ is $\GG$-invariant. 
\item For any $\theta \in \Theta^{adm} $ there exists $\Gg \in \GG$ such that $\theta \circ \Gg \in  \Theta^{adm,+}. $

\end{enumerate}
\end{claim}
\begin{proof}Clear.
\end{proof}
\begin{definition} We denote by $\mcP _a^{\bar w}(X)$ the space of functions $f$ such that $h_{\Gg ,f} $ is a polynomial of degree $\leq a$ on $L$ for all $\Gg \in \GG $.
\end{definition}
The group $T$ acts naturally on $X$, and on the spaces $\mcP ^{\bar w}_a(X)$ and $\mcP _a(X)$, and we have direct sum decompositions  
\[\mcP ^ { w}_a(X) =\oplus _{\theta \in \Theta}
\mcP ^ { w}_a(X) ^\theta\] 
and 
\[\mcP _a(X) =\oplus _{\theta \in \Theta}
\mcP _a(X) ^\theta.
\]
\begin{lemma}\label{almost}
 Let $\mZ\subset \mV$ be a homogeneous $k$-subvariety of degree $d$ and let $f:Z\to k$ a polynomial function of degree $ad$ which is a weakly polynomial function on $Z$ of degree $\leq a$. Then it is a restriction of polynomial function on $V(k)$ of degree $\leq a$. \end{lemma}
\begin{proof} Lemma \ref{almost} follows inductively from the following claim.

\begin{claim}\label{we} Let $f:Z\to k$ be a polynomial function of degree $\leq a$ which is a weakly polynomial of degree $<a$. Then $f$ is a polynomial of degree $<a$.
\end{claim}

\begin{proof} 
We can write $f$ as a sum $f=Q+f'$ where $\operatorname{f'}<a$ and $Q$ is homogeneous of degree $a$.
Since $f$ is weakly polynomial of degree $< a$ the function 
$Q$ is also weakly polynomial of degree $<a$. It is sufficient to show that $Q\equiv 0$.

Choose $z\in Z$. To show that $Q(z)=0$ consider the function $g$ on $k,g(t)=Q(tz)$. Since $Z$ is homogeneous  $tz \in Z$. Since $Q$ is homogeneous of degree $a$ we have $g(t)=ct^a$. On the other hand, since $Q$ is  weakly polynomial of degree $<a$ we see that $g(t)$ is a polynomial of degree $<a $. Since $a<q$ we see that $g\equiv 0$. So $Q(z)=g(1)=0$.
\end{proof}

\end{proof}
\begin{corollary} Let $f:X\to k$ be a weakly polynomial of degree $<a$ on $X$ which 
is a restriction of  polynomial function of degree $\leq ad$ on $V$. Then $f$ is a restriction of 
polynomial function of degree $\leq a$ on $V$. 
\end{corollary}
It is clear that for a proof  Theorem \ref{equality} it  suffices to show that $\mcP^{w}_a(X) ^\theta =\mcP _a(X) ^\theta$ for any $\theta \in \Theta$. This equality follows now immediately from the following statement. Fix $ \theta \in \Theta $. 
\begin{proposition}\label{Id} \leavevmode
For any weakly polynomial function $f:X\to k$ of degree $\leq a$  satisfying the equation 
$f(tx)=\theta (t)f(x), t\in T ,x\in X$ there exists a polynomial $P$ on $V$ of degree $\leq ad$ such that $f=P_{|X}$.
\end{proposition}

We start  the proof of Proposition \ref{Id} with a set of definitions.
Let $f:X\to k$ be a function such that  $f(tx)=\theta (t)f(x), t\in T ,x\in X,$ and such that $h_{\Gg ,f} $ are polynomial functions on $L$ of degree $\leq a$ for all $\Gg \in \GG$. 

\begin{definition}\leavevmode
\begin{enumerate}
\item  We write   $h,h_\Gg :L \to k$ instead of $h_{\text{Id},f}$ and $h_{\Gg ,f}$.
\item  $\nu :V \to L$ is the map given by $\nu (w_1,\dots ,w_n):=(\mu (w_1),\dots , \mu (w_n))$.

\item $W^0:=\{ w= (a^1,\dots ,a^d)\in W|a^i\in \D $ for $i\geq 2\}\subset W$.
\item $X^0:=(W^0)^n\cap X$.
\item For $ w=(a^1,\dots ,a^d) \in W$ we define $I( w):=\{ i,1\leq i\leq d|a^i\not \in \D \}$ and write $z(w):=\max (|I(w)|-1,0)$.
\item For $x=(w_1, \ldots, w_n)$ we write $z(x)=\sum _jz( w_j)$.
\item $Y_s:=\{ x\in X|z(x)\leq s\}, s\geq 0$.
\end{enumerate}
\end{definition}

\begin{claim}\label{tr}\leavevmode
\begin{enumerate}
\item $Y_0=\{ \Gg (W^0), \Gg \in \GG\}$.
\item For any  $x\in X^0$ there exist unique 
$t(x)\in T$ such that $x=t(x)\kk (\nu (x))$.
\item $f(x)=\theta (t(x))f(\kk (\nu (x)))$ for any $x\in X^0$.
\item For any $\Gg \in \GG$, $l\in L,$ we have $\nu (\Gg (l))=l$ where $\Gg : X\to X$ is as in Definition \ref{ga}.
\end{enumerate}
\end{claim}
\begin{proof}Clear.
\end{proof}

\begin{lemma}\label{s1} Let $f$ be a function on $X$ 
satisfying the conditions of Proposition \ref{Id} and such that $f_{|Y_0} \equiv 0$. Then $f\equiv 0$.
\end{lemma}
\begin{proof} It is clear that it is sufficient to prove the following statement
\begin{claim} Let $f$ be a function on $X$ 
satisfying the conditions of Proposition \ref{Id} and such that 
$f_{|Y_s} \equiv 0 ,s\geq 0$. Then $f_{|Y_{s+1}} \equiv 0$.
\end{claim}
\begin{proof} 
We want to show that $f(x)=0$ for all 
$x=(w_j)\in Y_{s+1}$. Since the restriction of $f$ on any line is a polynomial of degree $\leq ad$ it is clear that for a proof of the equality $f(x)=0$ it is sufficient to prove the following geometric statement.

\begin{claim}\label{x} There exists a line $N\subset X$ containing $x$ and such that  $|N\cap Y_s|>ad$.
\end{claim} 
\begin{proof} Let $x=(w_j), w_j =(x^1_j, \dots ,x^d_j),1\leq j \leq n$.
We start with the following observation.

\begin{claim} For any $x\in X\setminus Y_0, x=\{ x_j^i\}$ 
there exists $j_0$ such that  either there exist 
$(i_0,i_1), 1\leq i_0\neq i_1\leq d $ such that 
$ x_{j_0}^{i_0}=0, x_{j_0}^{i_1} \not \in \D$ or 
$\prod _ix^i_{j_0}\neq 0$ and $x_{j_0}^{i_0}\not \in \D$ for some $i_0, 1\leq i_0\leq d $.
\end{claim}
\begin{proof}Clear.
\end{proof}

Consider first the case when $x^{i_0}_{j_0}=0$ for some pair
$(i_0,j_0), 1\leq i_0\leq d, 1\leq j_0 \leq n $
and $x^{i_1}_{j_0}\not \in \D$ for some $i_1\neq i_0, 1\leq i_0\leq d   $. 
Let   $\alpha :k\to X$ be the map given by $ \alpha(c)=x^i_j(c)$ where 
$  x^i_j(c) = x^i_j$ if $(i,j)\neq (i_1,j_0)$ and $x^{i_1}_{j_0}(c)=c$. 
By construction $x\in N :=\text{Im}(\alpha) $ and $x^{i_1}_{j_0}(\D)\subset Y_s$ for $c\in \D$. Since $|\D| =q>ad $ we see that the line 
$N$ satisfies the conditions of Claim \ref{x}.

So we may assume  the existence of $j_0$ such $\prod _ix^i_{j_0}\neq 0$ and $x_{j_0}^{i_0}\not \in \D$ for some
 $i_0, 1\leq i_0\leq d $. To simplify notations we may and will assume  that $j_0=i_0=1$.

Since $x\in X$ there exists $j_1,2\leq j\leq n$ such that  $\prod _ix^i_{j_1}\neq 0$.  We may assume that $j_1=2$. It is clear that  either $x_2^i\in \D$ for all $i,1\leq i\leq d$ or  there exists  $i_1, 1\leq i_1\leq d $ such that $x_2^{i_1}\not \in \D$ in which case we may and will   assume that $i_1=1$. 

Let $a:=\prod _{i=2}^dx_1^i, b :=\prod _{i=2}^dx_2^i $ and $\beta :k\to X$ be the map given by $\beta (c):= x^i_j(c)$ where $x^1_1(c)=-bc, \ x^2_1(c)=ac-\sum _{j=3}^n\prod _{i=1}^dx^i_j$ and 
$  x^i_j(c) = x^i_j$ otherwise.
Let $a:=\prod _{i=2}^dx_1^i, b :=\prod _{i=2}^dx_2^i $ and $\beta :k\to X$ be the map given by $\beta (c):= x^i_j(c)$ where $x^1_1(c)=-bc, \ x^1_2(c)=ac-b^{-1}\sum _{j=3}^n\prod _{i=1}^dx^i_j$ and 
$  x^i_j(c) = x^i_j$ otherwise.
Let $N=\operatorname{Im}(\beta)$. Then $x\in N$ and $|N\cap Y_s|=q>ad$.
\end{proof}
\end{proof}
\end{proof}

\begin{lemma}\label{zero}  $\mcP_a^ {\bar w}(X)^\theta =\{0\}$ for any 
$\theta \not \in \Theta^{adm}$.
\end{lemma}

\begin{proof} As follows from Lemma \ref{s1} it is sufficient to show that 
$h_{\Gg}(f)\equiv 0$ for any 
$\theta \not \in \Theta^{adm},\Gg \in \GG$ and $f\in \mcP_a^ {\bar w}(X)^\theta $. Since 
$\theta \not \in \Theta^{adm}$
there exist $i,j,j',1\leq i\leq n, 1\leq  j,j'\leq d,$ such that $|\alpha _{j,j'}(\chi _i)|\geq a$. Choose $s\in S_d$ such $s(j)=1,s(j')=2$, and denote by $\ti s\in \GG$ the image of $s$ under the imbedding $S_d\ho \GG$ as the $i$-factor. After the replacing $f\to f\circ \ti s, \theta \to \theta \circ s$ we may assume that  $|\alpha _{1,2}(\chi _i)| > a$.

The functions $h_{\gamma}$  and  $h_{\gamma \circ  s}$ are weakly polynomial functions of degrees $\leq a$ on the linear space $L$. Therefore $h_{\gamma}$ and $h_{\gamma \circ s}$ are polynomial functions of degrees $\leq a$.

For any $l\in L$ such that $l_i\in \D$ we have
$h_{\gamma \circ  s}(l)=l_i^{\alpha _{1,2}(\chi _i)}h_{\gamma}(l)$. Since $\alpha _{1,2}(\chi _i)>a$ and $\alpha _{1,2}(\chi _i)\le m/2$, this is only possible if $h_{\gamma}=0$. 
\end{proof}

\begin{corollary} As follows from
 Claim \ref{plus} it is sufficient to prove Proposition \ref{Id} for 
$\theta \in \Theta ^{adm,+}$. 
\end{corollary}

\begin{definition}  
Let $P:V\to k$ be the polynomial given by
$$P(v)=\prod _{i=1}^n \prod _{j=1}^d(x_i^j)^{\alpha ^{1,j}(\chi _i)}h(\nu (v)), \quad v=( x_i^j).$$
\end{definition}
\begin{lemma}$\deg(P)\leq ad$.
\end{lemma}
\begin{proof}Let $b=\deg(h)$. It is sufficient to show that for any sequence $\bar e=(e(i))_{ 1\leq i\leq n}$,  $e(i)\in [1,d]$, we have 
$\sum _{i=1}^n \alpha ^{1,e(i)}(\chi _i) +b\leq a$.

Suppose there exists $\bar e$ such that $\sum _{i=1}^n \alpha ^{1,e(i)}(\chi _i) +b>a$.
Since $\theta \in \Theta _n^{adm}$, there exists a subset $I$ of $[1,n]$ such that 
$a<\sum _{i\in I} \alpha ^{1,e(i)}(\chi _i) +b\leq 2a$.

Let $ \Gg =(\bs _i)_{1\leq i\leq n}$, $\bs _i \in S_d,$ be such that  
$e(i)=\bs _i(1)$ for $i\in I$ and $\bs _i=Id$ if $i\not \in I$. Consider $h_{\Gg}:= \kk _\Gg ^\star (f)$. On one hand it is a polynomial of degree $\leq a$ on $L$. On the other $h_\Gg (l)=h(l)\prod _{i\in I}l_i^{ \alpha ^{1,e(i)}(\chi _i) }$. The inequalities 
$a<\sum _{i\in I} \alpha ^{1,e(i)}(\chi _i) +b\leq 2a$
imply that $h\equiv 0$.
\end{proof}

By construction $P_{|L}\equiv f _{|L}$.
Let $\bar f:=f-P$. Then $\bar f$ is weakly polynomial function of degree $\leq ad$ on $X$ vanishing on $L$ such that  $\bar f(tx)=\theta (t)f(x)$ for $t\in T, \ x\in X$. As follows from  Claim \ref{tr} we have $\bar f_{|X^0}\equiv 0$. It is clear 
that for a proof the part (1) of  Proposition \ref{Id} it is sufficient to show that $\bar f(x)=0$ for all $x\in X$. By Lemma \ref{s1} it suffices to prove the following lemma:

\begin{lemma}\label{van}$\bar f_{|Y_0} \equiv 0$.
\end{lemma} 

\begin{proof} Since $Y_0=\GG X^0=TL_\Gg$, it suffices to show that
 $f _{|L_\Gg} \equiv 0$ 
for all $\Gg \in \GG $. Let $h_\Gg :L\to k$ be given by 
$h_\Gg (l)=\bar f(\Gg (l))$. We have to show that $h_\Gg \equiv 0$. 
Since  $h_\Gg$ is a polynomial of degree $\leq ad$ it follows from Claim 
\ref{many} that it is sufficient to show that 
the restriction of $h_\Gg$ on $L_\D$ vanishes. But for any $l\in L_\D$ we
 have  $\Gg (l)=tl', t\in T,l'\in L$. Since $\bar f(tx)=\theta (t)f(x)$
we see that $h_\Gg (l)=0$.
\end{proof}

This completes the proof of Theorem \ref{equality}.

\end{proof}


\subsection{Proof of Theorem \ref{main1}}

\begin{proposition}\label{p-ext}  
There exists an effective bound $r=r(\bar d,a)$ such that if  $k$ is an 
admissible field which is either finite or algebraically closed, $W\subset V$ is a hyperplane 
then any weakly polynomial function on $X_{\bar P}, \ r_{nc}(\bar P)\geq r$ of degree $\leq a$ vanishing on $X_{\bar P}\cap W$ is a restriction of a polynomial on $V$ of degree $\leq a$.
\end{proposition}

As an immediate corollary we obtain:

\begin{corollary} \label{extension1}  There exists an effective bound $r=r(\bar d,a)$ such that 
the following holds for all admissible fields $k$ which are either finite or algebraically closed.  

Let  $\mX= \{v\in \mV|P_i(v)=0\}\subset \mV$ be 
a subvariety of degree  $\leq d$ and $\mW \subset \mV$ an affine subspace 
such that nc-rank of $\bar P_{|\mW}\geq r$. Let   $f$ be a weakly polynomial function on $X$
of degree $\leq a$ such that $f _{|X\cap W}$  extends to a polynomial on $W$ of degree $\leq a$. Then there  exists an extension $F$ of $f$ to 
a polynomial on $V$  of degree $\leq a$.
\end{corollary}

\begin{proof} Consider first the case when $W\subset V$ is a hyperplane.
By the assumption there exists an extension $R$ of the restriction $f|_W$.
Choose a linear projection  $s:V\to W $  and $f':X\to k$ by $f'(x)=f(x)-R(s(x))$. Then $f'$ is weakly polynomial function on $X$
of degree $\leq a$ such that $f' _{|X\cap W}\equiv 0$. As follows from Proposition \ref{p-ext}  
 there exists an extension of $f'$ to a polynomial $F'$ on  $V$ of degree $\leq a$. But then $F:=F'+R\circ s$ is an extension $F$ of $f$ to 
a polynomial on $V$  of degree $\leq a$.

In the case when the codimension of $W$ is $>1$ we choose a flag 
$$\mcF =\{W_0=W\subset W_1\dots \subset W_{\dim(V)-\dim(W)}=V\}, \dim (W_i)=\dim(W)+i,$$ 
and extend $f$ by induction in $i,1\leq i\leq \dim(V)-\dim(W) $ to a polynomial $F$ on $V$.
\end{proof}

\begin{remark} The choice of $F$ depends on a choice of flag $\mcF$ and on choices of projections used in the inductive arguments.
\end{remark}



\subsection{Proof of Proposition \ref{p-ext}}\label{jan-extension}

The key tool in our  proof of this proposition is a testing result from \cite{kz-uniform} which roughly says that any weakly polynomial function of degree $a$ on $X$ that is "almost" weakly polynomial of degree $<a$, namely it is a polynomial of degree $<a$ on almost all affine subspaces, is weakly polynomial of degree $<a$.  This part does not require $X$ to be of high rank. We use the assumption that $X$ is of  high rank to show (see Theorem \ref{B}) that almost any isotropic affine plane is contained in an isotropic three affine dimensional subspace that is not contained in  $l^{-1}\{0\}$. \\

We start by stating the testing result from \cite{kz-uniform}.
In \cite{KR} (Theorem 1)  the following description of degree $\le $ polynomials is given:

\begin{proposition}\label{kau-ron} Let $P:V \to k$. Then $P$ is a polynomial of degree $\le a$ if and only if
the restriction of $P$ to any affine subspace of dimension $l=\lceil \frac{a+1}{q-q/p}\rceil$ is a polynomial of degree $\le a$. 
\end{proposition} 

Note that when $a<q$ then $l\le 2$.  \\

In \cite{KR} the above criterion is used for polynomial testing over general finite fields. In \cite{kz-uniform} (Corollary 1.13) it is shown how the arguments in \cite{KR} can be adapted to polynomial testing within a subvariety variety $X \subset V$ (high rank is not required). 

\begin{proposition}[Subspace splining on $X$]\label{testing-lines}For any  $a,d, c>0$ there exists an $A=A(d,c,a) > 0$, depending polynomially on $c,d$ and exponentially on $a$, such that the following holds. 
Let $X \subset V(k)$ be  a  complete intersection of degree $d$, codimension $c$. Then any  weakly polynomial  function $f$ of degree $a$
such that the restriction of $f$ to $q^{-A}$-a.e  $l$-dimensional affine subspace, 
$l=\lceil \frac{a}{q-q/p}\rceil$, is a polynomial of degree $<a$, is weakly polynomial of degree $<a$.
\end{proposition}

\begin{proof}[Proof of Proposition \ref{p-ext}]
Let $V$ be a vector space and $l:V\to k$ be a non-constant affine function. For any subset $I$ of $k$ we denote $\mW _I=\{ v \in \mV |l(v)\in I\}$, so that  $\mW _b =\mW _{\{ b\} }$, for $b\in k$.
For a hypersurface $\mX \in \mV$ we write $\mX _I=\mX \cap \mW_I$.

\begin{lemma}\label{l} For any  finite subset $S\subset k$, any weakly polynomial function $f$ of degree $a$ on $X$ such that 
$f_{|X_S} \equiv 0$, and any $b\in k\setminus S$, there exists  a polynomial $Q$ of degree $\leq a$ on $V$ such that 
$Q_{|X_S}\equiv 0$ and
 $(Q-f)_{|X_b} \equiv 0$.
\end{lemma}

\begin{proof}
We start with the following result.

\begin{claim}Under the assumptions of Lemma \ref{l}, the restriction $f_{|X_b}$ is a weakly polynomial function of degree $\leq a-|S|$.
\end{claim} 

\begin{proof} Since $|k|>a$ it suffices to show that  for any plane $L\subset X_b$ the restriction $f_{|L}$ is a polynomial of degree  $\leq a-|S|$.

We first do the case when the field  $k$ is finite. Since by the part (4) of Theorem \ref{AC} the variety $\mX$ is a complete intersection it  follows from Proposition \ref{testing-lines}, there is a constant $A= A(d,a)$ such that it suffices to check the restriction $f_L$  on $q^{-A}$-almost any affine plane $L\subset X_b$ is a polynomial of degree $\leq a-|S|$.

As follows from Proposition \ref{line-plane}, for any $s>0$ there is an $r=r(d,s)$  such that if $X$ is of nc-rank $>r$ then for $q^{-s}$-almost any  affine plane $L\subset X_b$ there exists  an affine $3$-dim subspace  $M\subset X$ containing $L$ and such that $M\cap X_0\neq \emp$.  Then $M\cap X_t\neq \emp$ for any $t \in k$. Since $f$ is a weakly polynomial function of degree $a$, its restriction to $M$ is a polynomial $R$ of degree $\leq a$. Since the restriction of $R$ to  $l^{-1}(S) \cap M$ is identically zero, we see that $R=R'\prod _{s\in S}(l-s).$ 
Since $l|_L\equiv b$, we see that the restriction $f|_L$  is equal to $R'$, which is a 
is a polynomial of degree $\leq a-|S|$.
\end{proof}
Now we show that this claim implies Lemma \ref{l}. Indeed, assume that $f_{|X_b}$ is a weakly polynomial function of degree $\le a-|S|$.  It follows from the inductive assumption  on $a$, that there exists  a polynomial $Q'$ of degree $\leq a-|S|$ on $V$, such that $f_{|X_b}=Q'_{|X_b}.$  Let $Q:=\frac {Q'\prod _{s\in S}(l-s)}{\prod _{s\in S}(b-s) }$. Then $(f-Q)_{X_{S \cup \{b\}}}\equiv 0.$

For algebraically closed fields we follow  the same argument replacing Proposition \ref{testing-lines} with   Proposition \ref{asplining}, and Proposition \ref{line-plane}, with Theorem \ref{BC}. 
\end{proof}

 Proposition \ref{p-ext} follows from Lemma \ref{l} by induction. \\
\end{proof}

Now we can prove Theorem \ref{main}:

\begin{proof}[Proof of Theorem \ref{main} assuming Theorems \ref{const} and \ref{Jan},  and Corollary \ref{extension1}]

Let $\ti r$ be from Corollary \ref{extension1} and $r=r(a, \bar d):=\rho (\dim (W),d)$  from Theorem \ref{Jan}. As follows from Theorem \ref{const}   the subvarieties $\mX _n$ are of  nc-rank $\geq \ti r$ for $n\geq d\ti r$.

 Let $\mX \subset \mV$ be a subvariety of nc-rank $\geq r$.  By Theorem \ref{Jan}  there exists a linear 
map $\phi :\mW \to \mV$ such that $\mX _n=\{ w\in \mW|\phi (w)\in \mX\}$. Since $\mX _n$
satisfies $\star^k_a$,  Corollary \ref{extension1} implies that $\mX$ satisfies $\star ^k_a$. 
\end{proof}

\section{Nullstellensatz}

Let $k$ be a field and $V$ be a finite dimensional $k$-vector space. We denote by $\mV$ the 
corresponding $k$-scheme, and by $\mcP (V)$ the algebra of polynomial functions on $\mV$ defined over $k$.

For  a finite collection  $\bar P = (P_1,\ldots,P_c)$ of polynomials on $\mV$ we denote by $J(\bar P)$ the 
ideal in $ \mcP (V) $ generated by these polynomials, and by $\mX _{\bar P}$ the subscheme of $\mV$ defined by 
this ideal.

Given a polynomial $R \in \mcP(\mV)$, we would like to find out whether it belongs to the ideal 
$J(\bar P)$.  It is clear that the following condition is necessary for the inclusion $R\in J(\bar P)$.

\medskip
(N)   $R(x) = 0$ for all $k$-points $x \in X_{\bar P}(k) $.
\medskip

\begin{proposition}[Nullstellensatz]\label{Null} Suppose that the field $k$ is algebraically closed  and the scheme $\mX _{\bar P}$ is reduced. Then any polynomial $R$ satisfying the  condition $(N)$ lies 
in  $J(\bar P)$. 
\end{proposition}

We will show that the analogous result hold for $k=\mF _q$ if $\mX _{\bP}$ is of high $nc$-rank.

From now on we  fix a degree vector $\bd = (d_1,\ldots,d_c)$ and write
 $D:=\prod _{i=1}^cd_i$. We denote by $\mcP _{\bar d}(\mV)$ the space of  $\bar d$-families of polynomials  $\bP = (P_i)_{i=1}^c$ on $V$ such that $\operatorname{deg}(P_i) \leq d_i$. 

\begin{theorem}\label{main-null} There exists  and an effective bound $r(\bd)>0$ such that for any finite field $k=\mF _q$ with $q>aD$,  any family $\bar P$ of degrees $\bar d$ and nc-rank $\geq r(\bd) $  the following holds. Any polynomial 
$Q$ on $V$ of degree $a$ vanishing on $X_{\bar P}$ belongs to the ideal 
$J(\bP)$. 
\end{theorem}
\begin{proof}  Our proof is based on  the following {\it rough bound} (see \cite{hr}).
\begin{lemma}\label{rb} Let $\bar P=\{ P_i\}_{i=1}^c \subset \mF _q[x_1,\dots ,x_n]$ be a family of polynomials of degrees $d_i,1\leq i\leq c$ such that the variety 
$\mY :=\mX _{\bar P} \subset \mA ^n$ is of dimension $n-c$. Then
 $|\mY (\mF _q)|\leq q^{n-c}D$ where $D:=\prod _{i=1}^cd_i$.
\end{lemma}
For the convenience we reproduce the proof of this result.
\begin{proof} Let $F$  be the algebraic closure of $\mF _q$. Then $\mY (\mF _q)$ is the  intersection of $\mY$ with hypersurfaces $Y_j, 1\leq j\leq n$ defined by the equations $h_j( x_1,\dots ,x_n)=0$ where 
$h_j( x_1,\dots ,x_n)= x_j ^q-x_j$.

Let $H_1,\ldots,H_{n-c}$ be generic linear combinations of the $h_j$ with algebraically independent 
coefficients from an transcendental extension  $F'$ of $F$ and  $\mZ _1,...,\mZ _{n-c}\subset \mA ^n$ be the corresponding hypersurfaces. 

Intersect successively $\mY$ with $\mZ _1,\mZ _2,\dots ,\mZ _{n-c}$.   Inductively we see that for each $j \leq n-c$, each component $C$ of the intersection
 $\mY \cap  \mZ _1 \cap \dots \cap \mZ _j$ has dimension $n-c-j$.   Really passing from $j$ to $j+1$ for $j<n-c$ we have $\dim(C)=n-c-i>0$. So  not all the functions $h_j$ vanish on $C$. Hence by the genericity of the choice of linear combinations $\{H_j\}$ we see that 
$H_{j+1}$ does not vanish on $C$ and therefore $\mZ _{j+1}\cap C$ is 
of pure dimension $n-c-j-1$.    Thus  the intersection
 $\mY \cap  \mZ _1 \cap \dots \cap  \mZ  _{n-c}$ has dimension $0$.  By Bezout's theorem we see that  $|\mY \cap  \mZ _1 \cap \dots \mZ _{n-c}|\leq q^{n-c}D$. Since $\mY (\mF _q)=\mY  \cap  \mZ _1 \cap \dots \cap \mZ _n\subset \mX \cap  \mY _1 \cap \dots  \cap \mY _{n-c} $ we see that 
 $|\mY (\mF _q)|
\leq q^{n-c} D$.
\end{proof}

Now we can finish the proof of Theorem \ref{main-null}. Let $R \in \mF _q[x_1,\dots ,x_n] $ be a polynomial of degree $a$ vanishing on the set $\mX _{\bar P}(\mF _q)$. Suppose that $R$ does not lie in the ideal generated by the $P_i$, $1\leq i\leq c$. Let  $\mY :=\mX _{\bar P}$. Since $R$ vanishes on 
$\mX _{\bar P}(\mF _q)$ we have $\mY (\mF _q) = \mX _{\bar P}(\mF _q) $.

Since $\mX _{\bar P} $ is irreducible we have $\dim(\mY)=n-c-1$. As follows from \ref{rb} we have the upper bound
$|\mY (\mF _q)| \leq aD q^{n-c-1}$. On the other hand as follows from Theorem \ref{uniform} there exists an effective bound $r(\bd)>0$ such that the condition $r_{nc}(\bar P)\geq r(\bd) $  implies the inequality 
$|\mX _{\bar P}(\mF _q)|> q^{n-c}\frac{q-1}{q}$, which is a contradiction to $q>aD$. 
\end{proof}


\appendix \label{appendix-splining}

\section{An Algebro-Geometric analogue of Proposition  \cite{kz-uniform}} 
Proposition \ref{testing-lines} was proved in
\cite{kz-uniform}. In this Appendix  we present another approach leading to a proof of Proposition \ref{testing-lines} 
which (after some adjustments) provides a proof of an  Algebro-Geometric analogue of Proposition \ref{testing-lines}.

We start with two auxiliary results.

\begin{definition} Let $A$ be a finite set, $B\subset A\times A,  \ Z\subset A,$ and let $p_1, p_2: B\to A$ be the natural projections. For $x\in A$ we define  $B_x:=p_2(p^{-1})(x)\subset A$. Let $\beta (x):=|B_x|/|A|$.  
\begin{enumerate}
\item $B$ is $\delta $-dense if $\beta (x) \geq \delta , \ x\in A$.
\item A subset $C\subset A$ is $\ep$-invariant if $| B_z \cap C|/| B_z |\geq 1-\ep , \ z\in C$.
\end{enumerate}
\end{definition}

\begin{claim}\label{emp} If  $B$ is $\delta $-dense, $C\subset A$ is $\ep$-invariant, and $|C|/|A|< \delta (1-\ep) $,
then $C=\emp$.
\end{claim}

\begin{proof} Let $\alpha =|A|,\beta =|B| ,\Gg =|C| $. Since 
$$\bigcup _{z\in C}z \times B_z \cap C \subset C\times C,$$
we have $\Gg ^2\geq \alpha \Gg \delta (1-\ep)$. Since  $\Gg /\alpha < \delta (1-\ep)$, we see that $\Gg =0$.
\end{proof}

Let $k$ be a field and  $Q\in k[x,y,z]$ be a homogeneous polynomial of degree $a$. We denote by $\mC _Q\subset \mP ^2$ the subvariety  of $1$-dimensional subspaces $ L$ in $\mA ^3$, such that 
$Q_{ L}\equiv 0$, and write $C_Q:= \mC _Q(k)$. 

\begin{claim}\label{bounds}\leavevmode
\begin{enumerate}
\item The subvariety $\mC _Q\subset \mP ^2$  is a curve, and therefore the subset $\mP ^2 -\mC _Q$ of $\mP ^2 $ is Zariski dense.
\item If $k=\mF _q$ then
$|C_Q|\leq aq$.
\end{enumerate}
\end{claim}

\begin{proof} The part (1) is immediate. To prove (2) we fix a point $a\in \mP ^2(k) -C_Q $. For any line $M\subset \mP ^2(k)$ containing $a$, the subset $M\cap C_Q \subset M$ is defined by a non-zero polynomial of degree $a$. So $|M\cap C_Q|\leq a$. We see that  $|C_Q|\leq a(q+1)$.
\end{proof}

Let $R\in k[x,y,z]$ be a polynomial of degree $a$. We denote by $Q$ the degree $a$ homogeneous part of $P$.
We denote by $\mcL _3$ the variety  of affine lines $L$ in $\mA ^3$. For any $L\in \mcL _3$, we denote by $\ti L$ the corresponding $1$-dimensional linear subspace of $\mA ^3$.

Let $\mcL _R \subset \mcL _3$ be the subvariety of lines $L$, such that $P_L$ is of degree $<a$.

\begin{claim}\label{bounds1}\leavevmode
\begin{enumerate}
\item $L\in \mcL _P \lr \ti L\in C_Q$.
\item $|\mcL _P| / |\mcL |\leq a/q$.
\end{enumerate}
\end{claim} 
\begin{proof}The part (1) is immediate and the part (2) follows from the part (1) and Claim \ref{bounds}.
\end{proof}

Let $k$ be a field and $\mX :=\mX _{\bar P}$, where  $\bar P=( P_i)_{1\leq i\leq c},$ are polynomials on $k^N$ of degrees $\leq d_i$. We denote by $ \mA$ the variety  of affine lines in $\mX$,  by $ \mB$ the variety  of $3$-dimensional affine subspaces in $\mX$ and by $ \mE \subset \mA ^2$ the subvariety of pairs if lines $L,L'\in \mA$ such that there exists unique $W\in \mB$ containing both $L$ and $L'$. We have natural projections $p_1,p_2:\mE \to \mA$ and the surjection  $j:\mcE \to \mcB$.
We write $A:=\mA (k),B: =\mB (k) ,E: =\mE (k) $.
For any subvariety $\mZ \subset \mA$ we define $\ti Z:=p_2(p_1^{-1}(\mZ))$.

Let $f:X\to k$ be a weakly polynomial function of degree $\leq a$.
We denote by $C=C_f\subset A$ the subset of lines $L$ such that $\deg(f_L)=a$.  For $W\in B$ we denote by  $A(W)\subset A$ the set of lines $L\in W$, and write  $C(W):=A(w)\cap C$.

We will show that under the assumption that $r_{nc}(\bar P)\gg1$ and $k$ is either a finite or an algebraically closed field, the assumption that $C$ is {\it small} implies that $C=\emp$.

We first consider the case when $k=\mF _q$. In this case the size of $C$ is 
the number $|C|$. So  the statement to prove is Proposition \ref{testing-lines}.

\begin{proof}
We start with the following result. 

\begin{claim} $| C(W) |/| A(w) |\geq 1-a/q$  for any  $W\in B$ containing some $L\in C$.
\end{claim}
\begin{proof}Since $f$ is weakly of degree $\leq d$, its restriction $f_W$ on $W$ is a polynomial of degree $\leq d$. Since $L\in C$, the restriction  $f_W$ is of degree $d$ it follows from Claim \ref{bounds1} that $| C(W) |/| A(w) |\geq 1-a/q$.
\end{proof}

As follows from Lemma A.5 in \cite{{kz-uniform}} for any $\bar d =(d_1,\dots ,d_c)$ there exists $t(\bar d) $ such that 
the subset $B \subset A\times A $ is $q^{-t(d)}$-dense for any $X=X_{\bar P}$, where $P$ is a polynomial of degree $d$. So it follows from Claim \ref{emp} that in the case when   $A(d,L,a)= t(d) +1$ and $a/q\leq 1/2$, we have $C=\emp$.
\end{proof}

We consider now the case when $k$ is algebraically closed. In this case we 
measure the size of $C$ in terms of the codimension $\D (f)$  of the algebraic closure 
$\mZ$ of $C \in \mA$.

\begin{proposition}\label{asplining}
For any $a$ there exists a constant  $C(\bar d,a)$ such that the following holds. 

Any weakly polynomial function $f$ on $X$ of degree $\leq a$ such that $\D (f)\geq C(\bar d,a)$ is actually of degree $<d$.
\end{proposition}

\begin{proof} 
We start with the following result.
\begin{claim}\label{inc} For any $\bar d =(d_1,\dots ,d_c)$, there exists $t(\bar d,a)$, such that for any non-empty subvariety $\mZ \subset \mA$ of codimension $> t(\bar d,a)$
we have  $\dim(\ti \mZ)>\dim(\mZ)$.
\end{claim}

\begin{proof} It follows from Lemma A.5
in \cite{{kz-uniform}} (see Section 3.4) that $\dim(p_1^{-1}(L))\geq \dim (\mcA)-t(\bar d,a) $. Now the same arguments as in the proof of Claim \ref{emp} proves the inequality $\dim(\ti \mZ)>\dim(\mZ)$.
\end{proof} 

Now we can finish the proof of  Proposition \ref{asplining}. Let $D\subset E$ be the set of pairs of lines $L,L'\in E\cap (C\times C)$. As follows from Claim \ref{bounds} (1), the intersection $D\cap p_1^{-1}(L)(k)$ is dense in $p_1^{-1}(L) $. This implies that $\ti \mZ(k)\cap C$ is dense in $\ti \mZ $. Therefore $\ti \mZ \subset \mZ$. It follows now from Claim \ref{inc} that $\mZ =\emp$ and therefore $C =\emp $.

\end{proof}


\end{document}